\documentclass[11pt]{amsart}
\usepackage{fancyhdr}
\usepackage{amssymb,amsmath,amscd,amsthm}
\usepackage{verbatim}
\usepackage[T1]{fontenc}
\usepackage{tikz}
\usetikzlibrary{matrix}
\usepackage[all]{xy}
\usepackage{enumerate}
\newcommand{\Gauss}{\mathfrak{Gauss}}
\newcommand{\td}{\mathrm{td}}

\newcommand{\rez}{\mathrm{rez}}

\newcommand{\cl}{\mathrm{cl}}
\newcommand{\sgn}{\mathrm{sgn}}

\newcommand{\GW}{\mathrm{GW}}

\newcommand{\mMbar}{\overline{\mathcal{M}}}
\newcommand{\dvol}{\mathrm{dvol}}
\newcommand{\vol}{\mathrm{vol}}
\newcommand{\Obs}{\mathfrak{Obs}}
\newcommand{\obs}{\mathfrak{obs}}
\newcommand{\rem}{\mathrm{rem}}
\newcommand{\ev}{\mathrm{ev}}
\newcommand{\dbar}{\overline{\partial}}
\newcommand{\pd}{\mathcal{P}}

\newcommand{\ft}{\mathrm{ft}}

\newcommand{\exc}{\mathrm{exc}}
\newcommand{\codim}{\mathrm{codim}}

\newcommand{\RR}{\mathbb{R}}

\newcommand{\CC}{\mathbb{C}}
\newcommand{\ZZ}{\mathbb{Z}}

\newcommand{\PP}{\mathbb{P}}
\newcommand{\QQ}{\mathbb{Q}}

\newcommand{\mJ}{\mathcal{J}}
\newcommand{\mB}{\mathcal{B}}

\newcommand{\mL}{\mathcal{L}}

\newcommand{\mM}{\mathcal{M}}
\newcommand{\mV}{\mathcal{V}}
\newcommand{\mH}{\mathcal{H}}

\newcommand{\mE}{\mathcal{E}}
\newcommand{\mF}{\mathcal{F}}
\newcommand{\m}{\mathfrak{m}}
\newcommand{\FF}{\mathrm{FF}}
\newcommand{\mMuniv}{\mathcal{M}^{\mathrm{univ}}}

\newcommand{\GL}{\mathrm{GL}}

\newcommand{\coker}{\mathrm{coker}}
\newcommand{\End}{\mathrm{End}}
\newcommand{\id}{\mathrm{id}}
\newcommand{\im}{\mathrm{im}}

\newcommand{\Gl}{\mathfrak{Gl}}

\newtheorem{thm}{Theorem}
\newtheorem{thmmain}{Theorem}
\newtheorem{cormain}[thmmain]{Corollary}

\newtheorem{sch}{Scholium}
\newtheorem{prp}{Proposition}
\newtheorem{lma}{Lemma}
\newtheorem{dfn}{Definition}

\newtheorem{rmk}{Remark}
\newtheorem{cor}{Corollary}

\newcounter{daggerfootnote}
\newcommand*{\daggerfootnote}[1]{%
    \setcounter{daggerfootnote}{\value{footnote}}%
    \renewcommand*{\thefootnote}{\fnsymbol{footnote}}%
    \footnote[2]{#1}%
    \setcounter{footnote}{\value{daggerfootnote}}%
    \renewcommand*{\thefootnote}{\arabic{footnote}}%
  }
  
\author{Jonathan David Evans}
\title[Quantum cohomology of twistor spaces]{Quantum cohomology of twistor spaces and their Lagrangian submanifolds}
\address{University College London}
\email{j.d.evans@ucl.ac.uk}

\begin{document}
\begin{abstract}We compute the classical and quantum cohomology rings of the twistor spaces of 6-dimensional hyperbolic manifolds and the eigenvalues of quantum multiplication by the first Chern class. Given a half-dimensional totally geodesic submanifold we associate, after Reznikov, a monotone Lagrangian submanifold of the twistor space. In the case of a 3-dimensional totally geodesic submanifold of a hyperbolic 6-manifold we compute the obstruction term $\m_0$ in the Fukaya-Floer $A_{\infty}$-algebra of a Reznikov Lagrangian and calculate the Lagrangian quantum homology. There is a well-known correspondence between the possible values of $\m_0$ for a Lagrangian with nonvanishing Lagrangian quantum homology and eigenvalues for the action of $c_1$ on quantum cohomology by quantum cup product. Reznikov's Lagrangians account for most of these eigenvalues but there are four exotic eigenvalues we cannot account for.\end{abstract}
\maketitle
\tableofcontents

\section{Introduction}
The twistor space $Z$ of a Riemannian $2n$-manifold $M$ is the total space of the bundle of orthogonal complex structures on the tangent spaces of $M$. Reznikov \cite{Rez} wrote down a natural closed 2-form $\omega_{\rez}$ on twistor space and observed that if the curvature of $M$ satisfies a certain inequality then this 2-form is actually symplectic. He also demonstrated that above any totally geodesic submanifold of the middle dimension in $M$ there is an $SO(n)$-subbundle of the twistor space on which $\omega_{\rez}$ vanishes. We will call these Reznikov Lagrangians. For instance, when $M$ is the round 4-sphere the twistor space is the standard symplectic $\CC\PP^3$, an equatorial geodesic 2-sphere lifts to a Lagrangian $\RR\PP^3$ and an equatorial geodesic torus lifts to the Clifford torus.

An interesting class of manifolds for which the Reznikov curvature inequality holds are the hyperbolic $2n$-manifolds (that is compact quotients of hyperbolic $2n$-space by a discrete torsionfree subgroup of $SO^+(2n,1)$). These give twistor spaces which are of a very different character from that of the round $2n$-sphere. For instance they are non-K\"{a}hler (when $n>1$) by dint of their fundamental group being hyperbolic (see \cite{CT}). Nonetheless, as discovered in \cite{FP2} when $n\geq 3$ they are monotone, meaning that the first Chern class is positively proportional to the cohomology class of the Reznikov form; moreover Reznikov Lagrangians are monotone, meaning that the Maslov index of a bounding disc is positively proportional to its symplectic area. Monotone Lagrangians in monotone manifolds are amenable to modern pseudoholomorphic techniques without appeal to the massive machines under development to deal with the general case. What is even better is that there is a natural almost complex structure $J_-$, first discovered by Eells and Salamon \cite{ES}, which is compatible with $\omega_{\rez}$ in this very special hyperbolic setting. The $J_-$-holomorphic curves are in a close correspondence with branched minimal surfaces in $M$ by projection along the twistor fibration. This allows us to write down all the genus 0 holomorphic curves (see \cite{FP2}, Lemma 37) and all the discs with boundary on Reznikov Lagrangians and we have a hope of computing respectively the quantum cohomology and Lagrangian intersection Floer theory. Upon noticing this property of these Lagrangian submanifolds one feels like a fortunate astronomer who espies a charming and unusual galaxy perfectly angled so one can see the glory of its disc full on. Reznikov Lagrangians in the twistor space of a hyperbolic manifold are topologically much more complicated than the conventional examples of monotone Lagrangians: they are the principal frame bundles of hyperbolic $n$-manifolds.

If $n=2$ then the (6-dimensional) twistor space has $c_1=0$ and the Lagrangians are Maslov zero. This case is less amenable to simplistic techniques due to problems arising from transversality for multiple covers of Chern zero spheres and Maslov zero discs. Though the former are unlikely to cause major headaches I decided it would cloud the exposition and therefore I have restricted computation to the simplest case, $n=3$.
\begin{thmmain}
The small quantum cohomology of the twistor space of a hyperbolic 6-manifold $M$ with vanishing Stiefel-Whitney classes is\daggerfootnote{See erratum at the end of the paper.}
\[QH^*(Z;\Lambda)\cong H^*(M;\Lambda)[\alpha]/(\alpha^4=8\alpha\tau^*\chi+8q\alpha^2-16q^2)\]
where $\alpha=c_1(Z)$ and $\Lambda=\CC[q]$. Moreover, $c_1(\mH)^2=\alpha^2-4q$, $c_1(\mH)^3=\alpha^3-4\alpha q$. The twistor space is also uniruled.
\end{thmmain}
\begin{proof}
This follows directly from the classical cohomology ring computation in Section \ref{top-twist}, the computation of the 3-point Gromov-Witten contribution from twistor lines in Corollary \ref{gw-input} and Theorem \ref{crit-thresh} which proves there are no other quantum corrections. Uniruledness follows from the first part of Corollary \ref{gw-input}.
\end{proof}
The theorem probably holds with the assumption on Stiefel-Whitney classes replaced just by orientability, but this assumption allows us to represent various homology classes in the twistor spaces explicitly as submanifolds (see Section \ref{vis}) which simplifies the argument. Note that there is very little loss of generality by making this assumption: by {\cite[Corollary 2]{LaFRoy}}, any compact hyperbolic manifold admits a finite cover whose Stiefel-Whitney classes vanish. The use of complex coefficients is needed: the cohomology of the twistor space is additively isomorphic to the tensor product of the fibre and base with complex coefficients by the Leray-Hirsch theorem, but the obvious characteristic classes do not generate the integral cohomology of the fibre. We note the following interesting corollary.
\begin{cormain}\label{c1eig}
The action of $c_1(Z)$ on $QH^*(Z;\Lambda)$ by quantum cup product is given with respect to the basis $\tau^*y,\alpha\tau^*y,\alpha^2\tau^*y,\alpha^3\tau^*y$ (where $y$ runs over a basis for $H^*(M;\CC)$) by the matrix\(^\dagger\)
\[
\left(
\begin{array}{cccc}
0 & 0 & 0 & -16q^2\\
1 & 0 & 0 & 8\tau^*\chi\\
0 & 1 & 0 & 8q\\
0 & 0 & 1 & 0\\
\end{array}\right)
\]
i.e. when $y$ has positive degree this acts as the matrix with no $8\tau^*\chi$ entry, when $y=1$ this acts as the above matrix where $\tau^*\chi$ is replaced by the number $\chi(M)$. The characteristic polynomial of this action is\(^\dagger\)
\[\left(\lambda^4-8q\lambda^2-8\chi(M)\lambda+16q^2\right)\cdot\left(\lambda^4-8q\lambda^2+16q^2\right)^{D-1}\]
where $D=\dim_{\CC}H^*(M;\CC)$. The eigenvalues associated to the second factor are
\[\pm 2\sqrt{q}\]
each with multiplicity $2(D-1)$. The eigenvalues\(^\dagger\) associated to the first factor can be quite complicated.
\end{cormain}
To see the relevance of this corollary we recall some Floer theory. The book \cite{FOOO} explains how to associate to an arbitrary Lagrangian submanifold of a symplectic manifold a filtered $A_{\infty}$-structure on a suitable space of $\QQ$-chains. Since the Reznikov Lagrangians are monotone when $n\geq 3$ this theory simplifies considerably (see \cite{BC}). When $n=3$ the Reznikov Lagrangians bound holomorphic discs with Maslov index 2 and hence there could be a nontrivial ``obstruction'' term (the $\m_0$ operation in the filtered $A_{\infty}$-algebra).
\begin{thmmain}\label{obs-fooo}
If $\Sigma$ is an oriented totally geodesic submanifold of an oriented hyperbolic 6-manifold $M$ and $L_{\Sigma}$ denotes the Reznikov Lagrangian lift in the twistor space of $M$ then
\[\m_0=\pm 2\sqrt{q}[L_{\Sigma}]\]
Moreover
\[HF(L_{\Sigma},L_{\Sigma})=H_*(L;\CC[q^{1/2}]).\]
\end{thmmain}
It is well-known (see \cite{Aur}, Proposition 6.8) that the possible values for $\m_0$ on a monotone Lagrangian with nonvanishing self-Floer homology are the eigenvalues for the action of $c_1(Z)$ by multiplication on the small quantum cohomology. Indeed, the Fukaya category splits into summands indexed by these eigenvalues. It would be intriguing to find (or to rule out the existence of) monotone Lagrangians in the twistor space of a hyperbolic 6-manifold whose $\m_0$ equals one of the four ``exotic'' eigenvalues\(^\dagger\) from Corollary \ref{c1eig} involving the Euler characteristic of $M$.
\begin{rmk}
It may seem from a cursory reading of the paper that we do not make much use of the fact that $M$ is hyperbolic rather than just negatively-curved, but the computations with the linearised $\dbar$-operator assume that the natural metric and Eells-Salamon almost complex structure are an almost K\"{a}hler pair which happens precisely when $M$ is 4-dimensional and Einstein self-dual or else higher-dimensional and hyperbolic. Although we have not used this, it is interesting to note that the twistor spaces of hyperbolic $2n$-manifolds are some of the very few known non-K\"{a}hler examples of Ricci-Hermitian almost K\"{a}hler manifolds, that is to say the Ricci curvature form is a $(1,1)$-form. These metric occur as critical points of the Nijenhuis energy on the space of $\omega_{\rez}$-compatible almost complex structures.
\end{rmk}
\subsection{Outline of the paper}
We begin in Sections \ref{hom}-\ref{rezlag} by reviewing those aspects of the geometry and topology of twistor spaces and Reznikov Lagrangians which will be of use later in the paper. Section \ref{top-prelim} explains some of the (classical) topological tools we will use to compute both the classical and quantum cohomology rings of twistor spaces which are then applied in Section \ref{top-comp} to compute the classical cohomology ring of the twistor space of a hyperbolic 6-manifold and of the moduli space of `twistor lines', the $J_-$-holomorphic curves of lowest degree. We also explain how to push forward classes from the moduli space of marked twistor lines into the twistor space. The main tool is Borel-Hirzebruch theory for performing fibre integrals of characteristic classes along maps between classifying spaces.

In Section \ref{dfngw} we briefly recall the definition of Gromov-Witten invariants. Section \ref{linear} is dedicated to the study of the linearised $\dbar$-equation for $J_-$-holomorphic curves in twistor space and the crucial result is that we can construct elements of the cokernel bundle explicitly out of vector fields on $M$. This is used in Section \ref{gw-line} to compute the $k$-point Gromov-Witten contributions from the moduli space of twistor lines: one can compute the Gromov-Witten invariant by taking the Euler class of an obstruction bundle and pushing forward along the evaluation map. Section \ref{high-degree} calculates the remaining Gromov-Witten contributions needed for calculating the quantum cohomology in the case $n=3$ (when $M$ is a hyperbolic 6-manifold). The idea is once again that there is a nonvanishing section of the obstruction bundle, but care must be taken because the moduli space is no longer compact. An explanation of the main technical result is postponed to Section \ref{prfbigbadthm}.

In Section \ref{FF} we prove Theorem \ref{obs-fooo} using similar techniques.

\subsection{Acknowledgements}
It is my pleasure to acknowledge that this paper benefitted greatly from helpful conversations with Paul Biran, Joel Fine, Dusa McDuff, Jarek K\k{e}dra, Dmitri Panov (who long ago explained to me his own argument with Joel Fine for why these spaces should be uniruled), Dietmar Salamon and Ivan Smith. Like many symplectic geometers, I first encountered these spaces in the paper \cite{FP2}. During this work I was supported by an ETH Postdoctoral Fellowship.

\section{The homogeneous space $SO(2n)/U(n)$}\label{hom}
The homogeneous space $F:=SO(2n)/U(n)$ parametrises orthogonal complex structures on $\RR^{2n}$ equipped with the Euclidean metric and an orientation, i.e.
\[SO(2n)/U(n)=\{\psi\in\GL^+(\RR^{2n})|\psi^2=-1,\psi^T=-\psi\}\]
It comes equipped with a natural almost complex structure $j_F$ defined as follows. The tangent space $T_{\psi}SO(2n)/U(n)$ can be translated to a subspace $\pi_{\psi}$ passing through the origin in $\End(\RR^{2n})$ and $\psi$ acts by left multiplication on $\End(\RR^{2n})$ preserving $\pi_{\psi}$. In terms of coordinates $(x_1,\ldots,x_{2n})\in\RR^{2n}$, the result of applying $j_F$ to a tangent vector $v^k_{\ \ell}\in T_{\psi}F\cong\End(\RR^{2n})$ is
\[[j_F(v)]^j_{\ \ell}=\psi^j_{\ k}v^k_{\ \ell}\]
This is an integrable left-invariant almost complex structure and it is (tautologically) compatible with the left-invariant metric $g_F$ on $SO(2n)/U(n)$ induced by the Euclidean metric on $\RR^{2n}$. The corresponding 2-form $\omega_F(\cdot,j_F\cdot)=g_F(\cdot,\cdot)$ is symplectic so we have a natural K\"{a}hler triple $(g_F,j_F,\omega_F)$.

The exceptional isomorphisms in low dimensions give us
\[SO(4)/U(2)\cong\CC\PP^1,\ SO(6)/U(3)\cong\CC\PP^3\]
In general the $\ZZ$-cohomology ring is (\cite{TodaMimura}, Theorem 6.11)
\[\ZZ[e_2,e_4,\ldots,e_{2n-2}]/\{e_{4k}+\sum_{i=1}^{2k-1}e_{2i}e_{4k-2i}=0\}_{k\geq 1}\]
The tautological $U(n)$-bundle has Chern classes $c_i=2e_i$.

In particular $H_2(SO(2n)/U(n);\ZZ)=\ZZ$; an explicit generator is given by the subspace of complex structures preserving a given 4-plane and fixed on the orthogonal complement, namely
\[SO(4)\times U(n-2)/(U(2)\times U(n-2))\cong SO(4)/U(2)\]
In the case $n=3$ (when $F=\CC\PP^3$) this corresponds to a line. For any $n$, all holomorphic curves of degree one have this form and we will call them \emph{lines} by analogy. The space of lines $\mL(F)$ is identified with the Grassmannian
\[SO(2n)/SO(4)\times U(n-2)\]
and the space of lines $\mL_1(F)$ with a marked point is
\[SO(2n)/U(2)\times U(n-2)\]
Again by analogy we will write $H=e_2\in H^2(F)$, thinking of it as a hyperplane class.
\section{Twistor spaces}
\subsection{Setting}
The twistor space $Z$ of an oriented $2n$-dimensional Riemannian manifold $(M,g)$ is the total space of the \emph{twistor bundle} of $g$-orthogonal complex structures on the tangent spaces of $M$,
\[\begin{CD}
F @>>> Z\\
@. @VV{\tau}V\\
@. M
\end{CD}\]
with fibre $F_p=\tau^{-1}(p)=\{J\in\GL^+(T_pM)|J^2=-1,\ J^*=-J\}$. The fibre can be identified with the homogeneous space $SO(2n)/U(n)$.
\begin{rmk}
We will be concerned with the twistor spaces of compact, closed oriented hyperbolic $2n$-manifolds,
\[\Gamma\backslash SO^+(2n,1)/SO(2n)\]
where $\Gamma\subset SO^+(2n,1)$ is a cocompact discrete torsionfree subgroup. In this case we can write $Z$ globally (see \cite{FP2}, Section 2.3.3) as
\[\Gamma\backslash SO^+(2n,1)/U(n)\]
\end{rmk}
The twistor bundle inherits a connection $\nabla$ from the Levi-Civita connection of $g$. We will write $\mV\oplus\mH$ for the vertical-horizontal splitting of this connection and use this to define some extra geometric structure on $Z$. First of all we can define a metric using $\tau^*g$ on the horizontal spaces and $g_F$ on the vertical spaces. We write this
\[g_Z=g_F\oplus\tau^*g\]
We define almost complex structures on $T_{\psi}Z$ for $\psi\in\tau^{-1}(p)$ by
\[J_{\pm}=(\pm j_F)\oplus\tau^*\psi\]
(recall that $\psi$ is a complex structure on $T_pM$).
\begin{itemize}
\item The \emph{Atiyah-Hitchin-Singer almost complex structure} $J_+$ is sometimes integrable (if and only if either $n\geq 6$ and $g$ is conformally flat or $n=4$ and $g$ is self-dual),
\item The \emph{Eells-Salamon almost complex structure} $J_-$ is never integrable.
\end{itemize}
We will only be interested in $J_-$ because of the close relationship between $J_-$-holomorphic curves and minimal surfaces (see Section \ref{estc}). Using $g_Z$ and $J_{\pm}$ one can define compatible nondegenerate 2-forms $\omega_{\pm}$
\[\omega_{\pm}=(\pm\omega_F)\oplus(\tau^*\omega_{\psi}),\ \omega_{\psi}(\cdot,\psi\cdot)=g(\cdot,\cdot)\]
Reznikov observed that the 2-form
\[\omega_{\rez}=(-\omega_F)\oplus-\hat{R}(\omega_{\psi})\]
is closed (where $\hat{R}$ is the Riemann curvature acting on 2-forms). We observe that if $\hat{R}=\pm\id$ then $\mp\omega_{\rez}=\omega_{\pm}$ and hence is a $\mp J_{\pm}$-compatible symplectic form. Note that conditions for $J_-$ to tame $\omega_{\rez}$ are given in (\cite{FP}, Section 4.2).
\begin{rmk}
We will work with hyperbolic manifolds, for which $\hat{R}=-\id$ so $J_-$ is an $\omega_{\rez}$-compatible almost complex structure on the twistor space. The structure $J_+$ is integrable but there is no compatible symplectic form: the twistor space of a hyperbolic $2n$-manifold $M$ cannot be K\"{a}hler for $n>1$ by a theorem of Carlson and Toledo \cite{CT} since its fundamental group is equal to $\pi_1(M)$.
\end{rmk}
\subsection{Eells-Salamon twistor correspondence}\label{estc}
\begin{thm}
If $u:\Sigma\rightarrow Z$ is a $J_-$-holomorphic map into twistor space then its projection $\tau\circ u$ is (either constant or) a conformal harmonic map.
\end{thm}
If $u(\Sigma)$ is contained in a fibre (so that $\tau\circ u$ is constant) then we say $u$ is \emph{vertical}. Let $v:\Sigma\rightarrow M$ be a conformal immersion and define the \emph{normal twistor bundle} $\nu\rightarrow\Sigma$ to be the $SO(2n-2)/U(n-1)$-bundle over $\Sigma$ whose fibre $\nu_p$ at $p\in\Sigma$ is the space of orthogonal complex structures on the normal bundle to $v$ at $v(p)$. We can define a \emph{Gauss lift}
\[\Gauss(v):\nu\rightarrow Z\]
living over $v$. This map is defined in the obvious way so that 
\[\Gauss(v)(\nu_p)=\{\psi\in F_p|\psi(T\Sigma)=T\Sigma\}\]
\begin{thm}
The conformal immersion $v:\Sigma\rightarrow M$ is harmonic if and only if $\Gauss(v)$ is $J_-$-holomorphic.
\end{thm}
The construction of the Gauss lift extends to the case when $v$ has isolated branch points. Since weakly conformal harmonic maps $\Sigma\rightarrow M$ are precisely the branched minimal immersions \cite{OssermanETAL} that means we can always lift a weakly conformal harmonic map. We see that the (non-vertical) $J_-$-holomorphic curves in $Z$ are contained in the complex submanifolds which are the Gauss lifts of branched minimal immersions. In fact \cite{Rawn} if $v:\Sigma\rightarrow M$ is a minimal surface then there exists a $J_-$-holomorphic curve which projects to $v$. We loosely refer to the following as the Eells-Salamon twistor correspondence (Eells and Salamon proved it in the case $n=2$, where it really is a correspondence; Salamon proved it in general in \cite{Sal}).
\begin{thm}[Eells-Salamon twistor correspondence (ESTC)]
Let $(M,g)$ be an oriented $2n$-dimensional Riemannian manifold. Then (non-vertical) $J_-$-holomorphic curves in the twistor space $Z$ project to branched minimal surfaces in $M$ and any branched minimal surface arises this way. In the case $n=2$ the Gauss lift actually provides a bijection between these objects.
\end{thm}
Since we are looking at harmonic maps into hyperbolic manifolds, we recall the following useful theorem about harmonic maps into negatively curved manifolds, which captures the convexity of the harmonic map energy functional:
\begin{thm}[See Jost \cite{Jo}, Theorem 8.10.2]\label{convex}
Suppose $X$ is a compact Riemannian manifold with boundary and $Y$ is a complete Riemannian manifold with negative sectional curvatures. Given a map $f:\partial X\rightarrow Y$ and a homotopy class of maps $F:X\rightarrow Y$ such that $F|_{\partial X}=f$ there exists a unique harmonic map in this homotopy class.
\end{thm}
\subsection{Classification of $J_-$-holomorphic spheres}
In a hyperbolic manifold it is a classical fact that there are no minimal spheres. This is a consequence of Theorem \ref{convex} and the fact that $\pi_2(M)=0$ for a hyperbolic manifold. Convexity implies there is a unique minimal representative of any homotopy class, $\pi_2(M)=0$ implies that any such map is nullhomotopic and the constant map is the unique nullhomotopic minimal sphere. The ESTC now tells us that any $J_-$-holomorphic curve projects to a point via the twistor fibration $\tau$, that is:
\begin{prp}[\cite{FP2}, Lemma 37]
If $Z$ is the twistor space of a hyperbolic $2n$-manifold then the space of $J_-$-holomorphic spheres in $(Z,J_-)$ is precisely the space of vertical spheres, i.e. $j_F$-antiholomorphic spheres in the fibres of $\tau$.
\end{prp}
We will use this to calculate the genus 0 Gromov-Witten invariants of $Z$.

\subsection{Characteristic classes}
\subsubsection{First Chern class}\label{fcc}
Eells and Salamon showed (\cite{ES}, Proposition 8.1) that $c_1(Z,J_-)=0$ when $M$ is 4-dimensional. Fine and Panov (\cite{FP2}, Proposition 33) extended this to arbitrary dimensions for hyperbolic manifolds as follows
\begin{prp}[\cite{FP2}, Proposition 33]\label{FPchern}
The first Chern class of the twistor space of a hyperbolic $2n$-manifold is given by
\[c_1(Z,J_-)=-(n-2)c_1(\mH)=(n-2)[\omega]\]
where $\mH$ is the horizontal distribution considered as a complex rank $n$ bundle on $Z$ with the tautological complex structure $\psi$ at a point $(p,\psi)\in Z$.
\end{prp}
We see that there is a trichotomy:
\begin{itemize}
\item[$n=1$:] (\textit{General type}) $M$ is a hyperbolic 2-manifold. The twistor space is just $M$, the Reznikov 2-form is the area form and the Eells-Salamon almost complex structure is the unique $g$-orthogonal complex structure.
\item[$n=2$:] (\textit{Calabi-Yau}) $M$ is a hyperbolic 4-manifold. The twistor space is symplectically Calabi-Yau in the sense that $c_1=0$. Note that $Z$ cannot actually by Calabi-Yau in the standard sense: its fundamental group is isomorphic to $\pi_1(M)$ which is hyperbolic and hence cannot occur as $\pi_1$ of a K\"{a}hler manifold.
\item[$n\geq 3$:] (\textit{Fano}) Again, $Z$ cannot be K\"{a}hler but it is symplectically Fano.
\end{itemize}
The calculation of the first Chern class goes via the observation that the tangent bundle of twistor space splits ($U(n)$-equivariantly) as
\[\Lambda^2\mH^*\oplus\mH\]
where $\Lambda^2\mH^*$ is the vertical bundle and $\mH$ is the horizontal bundle, considered with the Eells-Salamon almost complex structure. Since $c_1(\mV)=c_1(\Lambda^2\mH^*)=-(n-1)c_1(\mH)$ and since $\mH|_F$ is the tautological $U(n)$-bundle over $F$ we deduce that
\[c_1(Z)|_F=-2(n-2)H\]
while
\[c_1(F)=-2(n-1)H\]
(Don't be put off by the minus signs: we're interested in $j_F$-\emph{antiholomorphic} curves!)
\subsubsection{Pontryagin classes}
Another advantage of working with the twistor spaces of hyperbolic manifolds is the following theorem of Chern \cite{Chern}
\begin{thm}\label{chern-pont}
An orientable hyperbolic manifold has vanishing Pontryagin classes.
\end{thm}
The Pontryagin class will crop up very often when we perform topological calculations later and this theorem will make our life significantly simpler.
\subsection{Some useful formulae}
The following is a useful formula from \cite{DavidovMuskarovActaMathHungarica}.
\begin{lma}\label{formulae}
If $X\in\mV$ is a vertical vector then $\nabla_X$ preserves the horizontal-vertical splitting, i.e.
\begin{gather*}
(\nabla_XY)^H=\nabla_X(Y^H)\\
(\nabla_XY)^V=\nabla_X(Y^V)
\end{gather*}
Moreover if $Y=\tilde{W}$ is the horizontal lift of a vector field $W$ on $M$ then
\[\nabla_X\tilde{W}=(\nabla_{\tilde{W}}X)^H\]
\end{lma}
\begin{proof}
The fact that the fibres are totally geodesic implies that $\nabla_X$ preserves the splitting. To prove the final formula note that
\[\nabla_X\tilde{W}-\nabla_{\tilde{W}}X=[X,\tilde{W}]\]
The bracket $[X,\tilde{W}]$ is vertical because $\tilde{W}$ is constant in the $X$-direction. The derivative $\nabla_X\tilde{W}$ is horizontal because $\nabla_X$ preserves the splitting. Equating horizontal and vertical components gives the formula.\end{proof}
Another useful observation concerns antiholomorphic curves in the twistor fibre $F$. If $\psi:\Sigma\rightarrow F$ is a $j_F$-antiholomorphic curve then in terms of local conformal coordinates on $\Sigma$
\begin{equation}\label{psieqn}\boxed{\partial_s\psi^j_{\ \ell}-\psi^j_{\ k}\partial_t\psi^k_{\ \ell}=0.}\end{equation}
\section{Reznikov Lagrangians}\label{rezlag}
The following construction follows Reznikov \cite{Rez}. We recall it for the reader's convenience and because it is of prime importance in what follows. In this section $M$ may be any $2n$-dimensional Riemannian manifold whose Reznikov 2-form is non-degenerate.
\subsection{Reznikov's construction}\label{rezcon}
Let $\Sigma$ be an $n$-dimensional submanifold of $M$ and consider the submanifold
\[L_{\Sigma}:=\left\{(p,\psi)\in Z|p\in\Sigma,\ \psi(T_p\Sigma)\perp T_pm\Sigma\right\}\]
living over $\Sigma$.
\begin{lma}
Suppose $\Sigma$ is totally geodesic. Then $TL_{\Sigma}$ contains the horizontal lift $\widetilde{T\Sigma}$ of $T\Sigma$.
\end{lma}
\begin{proof}
Since $\Sigma$ is totally geodesic, a $g$-exponential neighbourhood of a point $p\in\Sigma$ is contained in $\Sigma$. Parallel transport along geodesics emanating from $p$ preserves the splitting $T_pM=T_p\Sigma\oplus (T_p\Sigma)^{\perp}$ and hence preserves the condition for an endomorphism $\psi$ to lie in $L_{\Sigma}$. Thus the horizontal sections of $Z$ lying over $\Sigma$ are contained in $TL_{\Sigma}$.
\end{proof}
\begin{lma}
If $\Sigma$ is totally geodesic then $L_{\Sigma}$ is $\omega_{\rez}$-Lagrangian.
\end{lma}
\begin{proof}
The 2-form $\omega_{\rez}$ is block-diagonal with respect to the splitting $\mH\oplus\mV$ so it suffices to check $\omega_{\rez}|_{\widetilde{T\Sigma}}=0$ and $\omega_{\rez}|_{F\cap L_{\Sigma}}=0$ separately.

To prove horizontal vanishing of $\omega_{\rez}$, let $X$ and $Y$ be horizontal lifts of tangents to $\Sigma$. Then $\omega_{\psi}(X,Y)=g(X,-\psi(Y))=0$ since $\psi\in L_{\Sigma}$. Similarly $\omega_{\psi}$ evaluates to zero on pairs of vectors orthogonal to $\Sigma$, so $\omega_{\psi}$ is block-antidiagonal with respect to the splitting $T\Sigma\oplus(T\Sigma)^{\perp}$. Since $\Sigma$ is totally geodesic, the Riemann curvature tensor is block-diagonal with respect to this splitting and therefore $\hat{R}(\omega_{\psi})=\omega_{\rez}=0$.

To prove vertical vanishing, define the automorphism $\lambda:T_pM\rightarrow T_pM$ with respect to the splitting $T_p\Sigma\oplus (T_p\Sigma)^{\perp}$ as $1\oplus -1$ and define the involution of the twistor fibre $F_p=\tau^{-1}(p)$ by $\iota_{\Sigma}:\psi\mapsto-\lambda\psi\lambda^{-1}$ so that $F_p\cap L_{\Sigma}$ is the fibrewise fixed locus of $\iota$. The involution is $\omega_F$-antisymplectic, so that $F_p\cap L_{\Sigma}$ is Lagrangian in $F_p$.
\end{proof}
Clearly the proof implies that $L_{\Sigma}$ is also $\omega$-Lagrangian, but it is easier for $\omega_{\rez}$ to be non-degenerate than for $\omega$ to be closed, so this lemma is a stronger and more useful observation. We also note the following useful corollary.
\begin{cor}
There is a fibre-preserving antisymplectic involution $\iota_{\Sigma}$ of $\tau^{-1}(\Sigma)$ whose fixed point set is precisely $L_{\Sigma}$.
\end{cor}
We now seek to understand the fibre of $L_{\Sigma}$. This is naturally identified with $SO(n)$ as follows. Let $S(O(n)\times O(n))$ be the stabiliser of $T_p\Sigma$ in $SO(2n)$. This group acts transitively on $L_{\Sigma}\cap F_p$ and the stabiliser is $S\left(O(n)\times O(n)\right)\cap U(n)=O(n)_{\Delta}$ where $O(n)_{\Delta}$ denotes the diagonal. Therefore
\[L_{\Sigma}\cap F_p=S(O(n)\times O(n))/O(n)_{\Delta}\cong SO(n)\]
So when $n=2$, $L_{\Sigma}$ is an $S^1$-bundle; when $n=3$, $L_{\Sigma}$ is an $\RR\PP^3$-bundle.

In the hyperbolic case, the base space $\Sigma$ is a totally geodesic submanifold of a hyperbolic $2n$-manifold $M$ so its universal cover is a linear $n$-subspace of hyperbolic space, that is $\Sigma$ is a hyperbolic $n$-manifold.
\subsection{Holomorphic discs}\label{holdisc}
Let $\Sigma\subset M$ be a totally geodesic submanifold of a hyperbolic 4-manifold $M$. If we want to understand $J$-holomorphic discs with boundary on $L_{\Sigma}$ we must first understand the relative homotopy group $\pi_2(Z,L_{\Sigma};\ZZ)$ and the Maslov homomorphism on this group.
\begin{lma}\label{htpy-discs}
The homotopy classes of discs with boundary on $L_{\Sigma}$ are
\[\pi_2(Z,L_{\Sigma})\cong\begin{cases}
\ZZ^2&\text{ when }n=2\\
\ZZ&\text{ when }n\geq 3
\end{cases}\]
and the Maslov homomorphism is
\[\mu=2(n-2)\omega\]
\end{lma}
\begin{proof}
The homotopy calculation just uses the long exact sequence of the fibration of $L_{\Sigma}$ over $\Sigma$ with fibre $SO(n)$ and the facts that $\Sigma$ is hyperbolic and hence has no higher homotopy groups and that $\pi_1(\Sigma)$ injects into $\pi_1(M)$ because $\Sigma$ is totally geodesic.

It is clear from the long exact sequence that the generators for $\pi_2(Z,L_{\Sigma})$ when $n=2$ are the upper and lower hemispheres of the twistor fibre. When $n\geq 3$ the hemispheres of a `real' twistor line (i.e. one with boundary on the relevant $SO(n)$) are homotopic and one of them is enough to generate. The antisymplectic involution $\iota_{\Sigma}$ switches the two hemispheres of a real twistor line and reverses their orientations. In particular they have the same Maslov index. Gluing the two discs along their common boundary gives the twistor fibre $F$ and the Maslov indices add. However the Maslov index of this sphere is $2c_1(Z)\cdot [F]=2(n-2)\omega$, which gives the result.
\end{proof}

Using the ESTC we now describe all $J$-holomorphic discs with boundary on $L_{\Sigma}$.

\begin{prp}
Let $L_{\Sigma}$ be a Reznikov Lagrangian in the twistor space of a hyperbolic $2n$-manifold. Then the $J$-holomorphic discs $u$ with boundary on $L_{\Sigma}$ are all vertical.
\end{prp}
\begin{proof}
Let $u:(\Delta,\partial\Delta)\rightarrow (Z,L_{\Sigma})$ be a $J$-holomorphic disc and suppose it is not vertical. On the interior of $\Delta$ the local computation proving the ESTC implies that $u$ projects to a weakly conformal harmonic map $f=\tau\circ u:\Delta\rightarrow M$ with boundary on $\Sigma$. By Lemma \ref{htpy-discs}, $\partial f:\partial\Delta\rightarrow\Sigma$ is nullhomotopic in $\Sigma$. Let $F:\Delta\rightarrow\Sigma$ be a nullhomotopy of $\partial f$. Theorem \ref{convex} above ensures that there is a harmonic representative $\hat{F}$ in the homotopy class of $F$ with the same boundary values. The composition of $\hat{F}$ with the totally geodesic embedding $\Sigma\rightarrow M$ remains harmonic (\cite{ESam}, Section 5). However, a harmonic map into a negatively curved manifold is determined uniquely by its boundary values (again by Theorem \ref{convex}). Therefore $f$ is equal to $\hat{F}$.

The ESTC now implies that $u$ is contained in the Gauss lift $\Gauss(\hat{F})$ (since $f$ is weakly conformal and harmonic it has only branch point singularities in the interior of the disc and hence we define the Gauss lift to be the closure of the Gauss lift of the interior). In each fibre this consists of almost complex structures for which $\hat{F}_*T\Delta\subset T\Sigma$ is preserved. But Reznikov's Lagrangian lift of $\Sigma$ consists fibrewise of complex structures for which $T_p\Sigma$ is sent to its orthogonal complement. This implies that $u$ is contained in a subset of the twistor space disjoint from $L_{\Sigma}$, however the boundary of $u$ is supposed to lie on $L_{\Sigma}$.
\end{proof}

\section{Topological preliminaries}\label{top-prelim}
We recall some facts from topology which we will use in the computation of Gromov-Witten invariants.
\subsection{Cohomological pushforward}
We recall that it is possible to pushforward a cohomology class $\alpha$ along continuous maps of oriented compact manifolds by converting $\alpha$ into its Poincar\'{e} dual homology class, pushing that forward and then taking the Poincar\'{e} dual. If $f:X\rightarrow Y$ is the map and $\pd$ denotes the Poincar\'{e} duality map from cohomology to homology (supressing the manifold on which it takes place) then this means
\[f_!\alpha:=\pd^{-1}f_*\pd\alpha\]
The main properties of cohomological pushforward we will need are:
\begin{itemize}
\item $f_!(\alpha\cup f^*\beta)=(f_!\alpha)\cup\beta$.
\item If we have a pair of fibre bundles $p:E\rightarrow B$ and $p':E'\rightarrow B'$ with fibres $F$ and $F'$ respectively and oriented vertical tangent bundles then a commutative diagram
\[
\begin{CD}
F @>{f}>> F'\\
@VVV @VVV\\
E @>{e}>> E'\\
@V{p}VV @VV{p'}V\\
B @>>{b}> B'
\end{CD}
\]
such that the map $f$ has degree 1 implies the equality
\[b^*p'_!=p_!e^*\]
In particular a pullback of oriented fibre bundles satisfies this condition.
\end{itemize}

\subsection{Diagonal decompositions}\label{diagdecomp}
Let $\Delta^k:X\rightarrow X^k$ denote the diagonal map $x\mapsto (x,\ldots,x)$ and let $\{x_i\}_{i\in I}$ be an additive basis for $H^*(X;\CC)$. We will find a formula for the cohomology class $\Delta^k_!(1)\in H^*(X^k;\CC)\cong H^*(X;\CC)^{\otimes k}$. This is what we call a \emph{decomposition of the diagonal}. We first introduce some notation. Let
\[g_{ij}=\int x_i\cup x_j\]
be the Poincar\'{e} pairing and $g^{ij}$ its inverse matrix (so $g_{ab}g^{bc}=\delta_a^{\ c}$). Denote by $C_{jk}^{\ \ i}$ the coefficients of the cup product
\[x_j\cup x_k=C_{jk}^{\ \ i}x_i\]
If we think of $(H^*(X;\CC),g)$ as an inner product space then we can raise and lower indices with $g$ and we see that
\[C_{jk\ell}=g_{i\ell}C_{jk}^{\ \ i}=\int x_j\cup x_k\cup x_{\ell}\]
Note that we must be careful with the order of indices since cup product is only graded-commutative. Finally, define the \emph{Poincar\'{e} amplitude}
\[P^{i_1\cdots i_k}=g^{i_1b_1}g^{c_1i_2}g^{a_{k-2}i_k}\prod_{m=2}^{k-2}g^{c_mi_{m+1}}g^{a_{m-1}b_m}\prod_{m=2}^{k-2}C_{a_mb_mc_m}\]
Mnemonically, we can think of this as the `Feynman amplitude' associated to the diagram
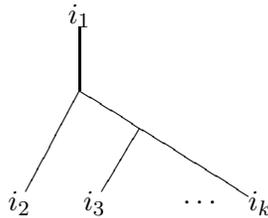
\[
\xy
(8,5)*{i_1};(8,-5)*{} **\dir{-};
(8,-5)*{};(32,-20)*{i_k} **\dir{-};
(8,-5)*{};(0,-20)*{i_2} **\dir{-};
(16,-10)*{};(10,-20)*{i_3} **\dir{-};
(24,-20)*{\cdots};
\endxy
\]
where the incoming edges to the $m$-th interior vertex (from the left) are labelled $b_m,c_m,a_m$ (clockwise from the topmost). Here, the Feynman rules associate a propagator $g^{pq}$ to an edge connecting downwards from $p$ to $q$ and a cubic interaction $C_{a_mb_mc_m}$ to the $m$-th interior vertex.
\begin{lma}\label{lma-diagdecomp}
\begin{align*}
\Delta^2_!(1)&=\sum_{ij}g^{ij}x_i\otimes x_j\\
\Delta^k_!(1)&=\sum_{i_1,\ldots,i_k}P^{i_1\cdots i_k}x_1\otimes\cdots\otimes x_k\mbox{ when }k\geq 3
\end{align*}
\end{lma}
\begin{proof}
The first equation is just the Alexander-Whitney formula. The second will follow by induction. Observe that $\Delta^{k+1}$ factors as
\[X\stackrel{\Delta^k}{\xrightarrow{\hspace*{1cm}}} X^k\stackrel{\id^{k-1}\otimes\Delta^2}{\xrightarrow{\hspace*{2cm}}} X^{k-1}\times X^2\]
Assuming inductively that the lemma holds for $\Delta^k$, we get
\[\Delta^{k+1}_!(1)=P^{i_1\cdots i_{k-1}b_{k-1}}x_{i_1}\otimes\cdots\otimes x_{i_{k-1}}\otimes\Delta^2_!(x_{b_{k-1}})\]
Since $\Delta^m$ is a diagonal embedding,
\[\Delta^m_!(x_{b_{k-1}})=\Delta^m_!(1)\cup (x_{b_{k-1}}\otimes 1^{\otimes m-1})\]
so (with propitious index naming)
\begin{align*}
\Delta^2_!x_{b_{k-1}}&=g^{a_{k-1}i_{k+1}}(x_{a_{k-1}}\cup x_{b_{k-1}})\otimes x_{i_{k+1}}\\
&=g^{a_{k-1}i_{k+1}}C_{a_{k-1}b_{k-1}c_{k-1}}g^{c_{k-2}i_k}x_{i_k}\otimes x_{i_{k+1}}
\end{align*}
which completes the induction step. Deriving the case $k=3$ from the Alexander-Whitney formula is elementary.
\end{proof}
In the sequel we will frequently use Poincar\'{e} amplitudes of different spaces and to distinguish them we will sometimes use decorations e.g. $P_X^{i_1\cdots i_k}$.

\begin{rmk}
We observe that if $\pi:A\rightarrow B$ is a fibre bundle satisfying the hypotheses of the ($\CC$) Leray-Hirsch theorem (i.e. there exist $\CC$-cohomology classes $\{z_i\}_{i=1}^N$ on $A$ which pull back to give a basis of the $\CC$-cohomology of the fibre) then the diagonal decomposition for $A$ has the form
\[\sum_{z_{i_1},\ldots,z_{i_k}}A^{i_1\ldots i_k}(z_{i_1}\otimes\cdots\otimes z_{i_k})\cup\pi^*\Delta_!^k(y_{i_1\ldots i_k})\]
where $y_{i_1\ldots y_k}$ is an element of $H^*(B;\CC)$ (we have abusively written $\pi:A^k\rightarrow B^k$). To see this, form the pullback
\[
\begin{CD}
A^{\otimes k} @>{\delta}>> A^k\\
@V{\tilde{\pi}}VV @VV{\pi}V\\
B @>>{\Delta^k}> B^k
\end{CD}
\]
The pushforward along $A\rightarrow A^k$ factors through this map. Note that $\tilde{\pi}:A^{\otimes k}\rightarrow B$ also satisfies the hypotheses of the Leray-Hirsch theorem since any cohomology class on the fibre $F^k$ is just a pullback of a product of classes from $A^k$. Therefore any class in $H^*(A^{\otimes k};\CC)$ can be written
\[(\delta^*c)\cup\tilde{\pi}^*y\]
pushing this class forward gives
\[\delta_!(\delta^*c\cup\tilde{\pi}^*y)=c\cup\delta_!\tilde{\pi}^*y=c\cup\pi^*\Delta_!^k(y)\]
which is of the required form.
\end{rmk}
\subsection{Borel-Hirzebruch theory}
Let $G$ be a compact connected Lie group and $H\subset G$ a closed subgroup containing a maximal torus $T$ of $G$. Borel-Hirzebruch theory is a means of calculating fibre integrals along bundle projections and takes as its starting point the diagram
\[
\xy
{\ar_{B\jmath} (-2,0)*{}; (-2,-10)*{}};
(-2,2)*{BT};(12,-12)*{BG};(-2,-12)*{BH};
{\ar^{B\kappa} (0,0)*{}; (10,-10)*{}};
{\ar_{B\iota} (2,-12)*{}; (8,-12)*{}};
\endxy
\]
where $T\stackrel{\jmath}{\rightarrow}H\stackrel{\iota}{\rightarrow}G$ are the inclusions and $\kappa=\iota\circ\jmath$. Fix $\Theta$ a positive system of roots of $G$ and $\Psi\subset\Theta$ a positive system of roots for $H$. We can think of these roots as element of the cohomology $H^*(BT;\CC)$. The Borel-Hirzebruch formula (\cite{BH}, Section 22.6) tells us that
\begin{equation}\label{borel-hirz1}\boxed{(B\iota)_!x=(B\kappa)_!\left((B\jmath)^*x\cup\td\right)}\end{equation}
where $\td$ is the Todd class
\[\prod_{\alpha\in\Psi}\frac{\alpha}{1-e^{-\alpha}}\in H^{**}(BT;\CC)\]
of the vertical tangent bundle to the $H/T$-bundle $BT\rightarrow BH$. Here $H^{**}$ denotes the direct product $\prod_{i\geq 0}H^i$ as opposed to the direct sum. Identifying $H^*(BG;\CC)$ as the Weyl-invariant subspace of $H^*(BT;\CC)$ via $(B\kappa)^*$, the pushforward $(B\kappa)_!$ can be computed using \cite{BH}, Equation (6):
\begin{equation}\label{borel-hirz2}\boxed{(B\kappa)_!y=\frac{\sum_{\sigma\in W(G)}\sigma(y)\sgn(\sigma)}{\prod_{\alpha\in\Theta}\alpha}}\end{equation}
where $\sgn(\sigma)$ denotes the determinant of $\sigma\in W(G)$ considered as a matrix acting on $\mathfrak{g}$.

We will apply Borel-Hirzebruch theory to the following situation. Let $T\subset H\subset G$ be as before and let $\pi:E\rightarrow B$ be a bundle of homogeneous spaces over a closed manifold arising as a pullback
\[
\begin{CD}
G/H @= G/H\\
@VVV @VVV\\
E @>{\cl_E}>> BH\\
@V{\pi}VV @VV{B\iota}V\\
B @>>{\cl_B}> BG
\end{CD}
\]
We can evaluate fibre integrals (pushforwards along $\tau$) of polynomials in the characteristic classes of the $H$-bundle over $Z$, for if $c$ is such a polynomial then
\[\tau_!\cl_E^*c=\cl_B^*(B\iota)_!c\]
This is useful for computing the ring structure on the cohomology of $Z$. Henceforth we assume that there exists an additive basis $C$ for the cohomology of $G/H$ coming from pulling back (along the inclusion $G/H\rightarrow Z$) the characteristic classes of the tautological $H$-bundle over $Z$. In this setting the Leray-Hirsch theorem implies that the cohomology $H^*(E;\CC)$ is a free $H^*(B;\CC)$-module generated by these characteristic classes. In particular if $C=\{z_i\}_{i\in I}$ and $\{y_j\}_{j\in J}$ is a basis for the cohomology of $B$ then any cohomology class $\beta$ can be written
\[\beta=\sum_{i\in I,j\in J}\beta^{ij}z_i\cup\tau^*y_j\]
and
\[\tau_!(\beta\cup z_{i_1}\cup\cdots\cup z_{i_p})=\sum_{i\in I,j\in J}\beta^{ij}\tau_!(z_i\cup z_{i_1}\cup\cdots\cup z_{i_p})\cup y_j\]
The fibre integrals on the right-hand side can be done using Borel-Hirzebruch theory and this gives a system of linear equations for the coefficients $\beta^{ij}$. This can be used to compute the cup product coefficients because if we take $\beta=(z_a\cup\tau^*y_b)\cup(z_c\cup\tau^*y_d)$ then
\[\beta^{ij}=C_{ab,cd}^{\ \ \ \ \ \ ij}.\]

\section{Topological computations}\label{top-comp}
We now perform some fibre integrals and computations of cohomology rings which will prove useful in the sequel.
\subsection{The twistor space}\label{top-twist}
The first bundle of interest is the twistor bundle
\[\tau:Z\rightarrow M\]
In this case $\cl_M$ classifies the ($SO(2n)$) frame bundle, $\cl_Z$ classifies the almost complex horizontal $U(n)$-bundle $\mH$ and the fibre is $SO(2n)/U(n)$. We may take as a system of positive roots of $SO(2n)$ and $U(n)$ respectively
\begin{align*}
\Psi&=\{x_i-x_j|i<j\}\\
\Theta&=\{x_i-x_j,x_i+x_j|i<j\}
\end{align*}
We now compute some fibre integrals and the cup product structure on $Z$ in the cases $n=3$ where the fibre is diffeomorphic to $\CC\PP^3$ and the cohomology of the fibre is therefore generated by powers of the first Chern class. When $n$ is larger the computations become unwieldy.

We have
\begin{align*}
(B\iota)_!c_1^3&=8&(B\iota)_!c_2c_1&=4\\
(B\iota)_!c_1^4&=0&(B\iota)_!c_2c_1^2&=0\\
(B\iota)_!c_1^5&=16p_1&(B\iota)_!c_2c_1^3&=4p_1\\
(B\iota)_!c_1^6&=64\chi&(B\iota)_!c_2c_1^4&=32\chi
\end{align*}
and, since $(B\jmath)^*c_3=x_1x_2x_3=(B\kappa)^*\chi$ is Weyl-invariant, we have for any $\ell\in H^*(BU(3);\CC)$
\[(B\iota)_!(c_3\cup\ell)=\chi\cup(B\iota)_!\ell\]
Now by Leray-Hirsch we can write any cohomology element of $Z$ as a linear combination of powers of $c_1(\mH)$ with coefficients in the cohomology of $H^*(M;\QQ)$. In particular
\[c_1(\mH)^4=\tau^*(\alpha)+c_1(\mH)\cdot\tau^*(\beta)+c_1^2\tau^*(\gamma)+c_1^3\tau^*(\delta)\]
Fibre integration yields
\begin{align*}
\delta&=0\\
\gamma&=2p_1=0\\
\beta&=8\chi
\end{align*}
which determines the ring structure completely as
\[H^*(Z;\CC)=H^*(M;\CC)[c_1(\mH)]/\left(c_1(\mH)^4=8c_1(\mH)\tau^*\chi\right)\]
By similar means we can express $c_2(\mH),\ c_3(\mH)$ in terms of $c_1(\mH)$:
\[c_2(\mH)=\frac{1}{2}c_1(\mH)^2,\ c_3(\mH)=\tau^*\chi\]
Let $z_i=c_1(\mH)^i$ and $y_j$ be an additive basis for $H^*(M;\CC)$ with $y_0=1$. In terms of the basis $z_i\cup\tau^*y_j$ for the cohomology of $Z$ we have Poincar\'{e} pairing
\[g^Z_{ia,jb}=8g^M_{ab}\delta_{i+j,3}+64\chi(M)\delta_{a,0}\delta_{b,0}\delta_{i+j,6}\]
with inverse
\[g^{ia,jb}_Z=\frac{1}{8}g^{ab}_M\delta^{i+j,3}-\chi(M)\delta^{y_a,\vol}\delta^{y_b,\vol}\delta^{i+j,0}\]
where $\vol$ is the (unit) volume form on $M$ (assumed to be a part of our basis). We also have
\[C^Z_{ia,jb,kc}=8\delta_{i+j+k,3}C^M_{abc}+64\delta_{i+j+k,6}\chi(M)\delta_{a,0}\delta_{b,0}\delta_{c,0}\]
From this and Lemma \ref{diagdecomp} we deduce
\begin{cor}[Formulae for the first few diagonal decompositions]\label{cor-diagdecomp-Z}
\begin{align*}
\Delta^2_!&=\frac{1}{8}\sum_{i=0}^3(c_1(\mH)^i\otimes c_1(\mH)^{3-i})\cup\tau^*\Delta_!^2(1_M)-\chi(M)\tau^*\Delta_!^2(\vol_M)\\
\Delta^3_!&=\frac{1}{64}\sum_{\substack{0\leq i,j,k\leq 3\\ i+j+k=6}}(c_1(\mH)^i\otimes c_1(\mH)^j\otimes c_1(\mH)^k)\cup\tau^*\Delta_!^3(1_M)\\
&\ \ \ \ \ \ \ \ \ +\frac{\chi(M)}{8}\left(c_1(\mH)^{\otimes 3}-\{c_1(\mH)^3\otimes 1\otimes 1\}\right)\cup\tau^*\Delta_!^3(\vol_M)
\end{align*}
where $\vol_M$ is a volume form with $\int_M\vol_M=1$, $1_M$ denotes the fundamental class of $M$, the curly brackets
\[\{z_{i_1}\otimes\cdots\otimes z_{i_k}\}\]
denote
\[\sum_{\sigma\in S_k/\mathrm{Stab}(z)}z_{i_{\sigma(1)}}\otimes\cdots\otimes z_{i_{\sigma(k)}}\]
and $\mathrm{Stab}(z)$ denotes the subgroup of permutations acting trivially on $(z_{i_1},\ldots,z_{i_k})$ (which is nontrivial when $z_{i_m}=z_{i_n}$ for some $m\neq n$). We have abused notation slightly and written $\Delta^k$ also for the diagonal inclusion of $M$ into $M^k$ and $\tau:Z^k\rightarrow M^k$ for the product of projections.
\end{cor}
\subsection{The moduli space of twistor lines}\label{top-prelim-twistorlines}
Recall that the holomorphic curves of minimal degree in the twistor fibre over $p$ are the \emph{twistor lines} consisting of complex structures preserving a fixed 4-plane in the tangent space $T_pM$ and equal to some fixed complex structure on its orthogonal complement. We write $\mL$ for the space of twistor lines and $\lambda:\mL\rightarrow M$ for the projection taking a line in the fibre $F_p$ to the point $p$. The fibre of $\lambda$ is the homogeneous space $\mL(F)=SO(2n)/SO(4)\times U(n-2)$ of lines in a single twistor fibre. We see that $\lambda:\mL\rightarrow M$ is the pullback of the tautological $SO(2n)/SO(4)\times U(n-2)$-bundle over $BSO(2n)$ along the classifying map for $TM$, so it fits into our Borel-Hirzebruch setup.

We will compute the cohomology ring in the case $n=3$. This will later be used to find the Poincar\'{e} amplitudes of a decomposition of the diagonal: these amplitudes will occur as coefficients in our formula for the Gromov-Witten class associated to the moduli space of twistor lines.

When $n=3$, the fibre of $\lambda$ is diffeomorphic to the Grassmannian
\[SO(6)/SO(4)\times SO(2)\]
and so its cohomology ring is
\[\CC[e,t]/(e^2=t^4,\ et=0)\]
where $e$ is the Euler class of the $SO(4)$-bundle and $t$ is the first Chern class of the $U(1)$-bundle. To see the relations, observe that: $e^2$ is Poincar\'{e} dual to the point in the Grassmannian representing the 4-plane orthogonal to a generic pair of vectors in $\RR^6$; $t^2$ is Poincar\'{e} dual to the point in the Grassmannian representing the 4-plane spanned by four generic vectors and $et$ is represented by the (generically empty) cycle of 4-planes which contain a given vector and are orthogonal to another. We need to calculate the fibre integrals
\[\lambda_!e^at^b\]
(where $e$ and $t$ now denote the corresponding characteristic classes on $\mL$). The only interesting ones are
\begin{align*}
\lambda_!t^4=\lambda_!e^2&=2\\
\lambda_!t^6&=2p_1(=0\mbox{ when }M\mbox{ is hyperbolic by Theorem \ref{chern-pont}})\\
\lambda_!et^5=\lambda_!e^3t&=2\chi
\end{align*}
These yield the cohomology ring
\begin{align*}H^*(\mL;\CC)&=H^*(M;\CC)[e,t]/(e^2=t^4,\ et=\lambda^*\chi)\end{align*}
when $M$ is hyperbolic.

Let $z_i$ run over the set $\{1,t,t^2,t^3,t^4,e\}$ and $y_j$ be a basis for $H^*(M;\CC)$. Then we see the Poincare pairing is
\[\int_{\mL}(z_i\cup\lambda^*y_j)\cup (z_k\cup\lambda^*y_{\ell})=2g^M_{j\ell}(\delta_{z_iz_k,t^4})\]
We also have
\[C_{ia,jb,kc}=2\delta_{i+j+k,4}C^M_{abc}+2\chi(M)\{\delta_{j+k,5}\delta_{i,e}\}\delta_{a0}\delta_{b0}\delta_{c0}\]
(where $\{\}$ denotes the sum over permutations as before) and so we deduce
\begin{cor}[Formulae for the first few diagonal decompositions]\label{cor-diagdecomp}
\begin{align*}
\Delta^2_!&=\frac{1}{2}\left(\{1\otimes t^4\}+\{t\otimes t^3\}+t^2\otimes t^2+e\otimes e\right)\cup\lambda^*\Delta_!^2(1_M)\\
\Delta^3_!&=\frac{1}{4}\left(\{t^4\otimes e\otimes e\}+\sum_{\substack{0\leq i,j,k\leq 4\\ i+j+k=8}}(t^i\otimes t^j\otimes t^k)\right)\cup\lambda^*\Delta^3_!(1_M)+\\
&\ \ \ +\frac{1}{4}\sum_{\substack{0\leq i<j\leq 3\\ i+j=3}}\{e\otimes t^i\otimes t^j\}\cup\lambda^*\Delta_!^3(\vol_M)
\end{align*}
where we have used the notation of Corollary \ref{cor-diagdecomp-Z}.
\end{cor}
\subsection{The evaluation map}\label{ev-pushfwd}
We denote by $\mL_1$ the moduli space of twistor lines with one marked point on the domain. Note that this has a forgetful map $\ft_1:\mL_1\rightarrow\mL$ which exhibits it as an $SO(4)/U(2)$-bundle and an evaluation map $\ev:\mL_1\rightarrow Z$ which sends a twistor line and a point on the twistor line to the corresponding point in $Z$. From the first description we see that $\mL_1$ is the pullback of the tautological $SO(2n)/U(2)\times U(n-2)$ over $BSO(2n)$ along the classifying map for $TM$. The evaluation map is induced by the inclusion $U(2)\times U(n-2)\rightarrow U(n)$.

We will need to compute cohomological pushforwards along $\ev$ of certain classes $\ft_1^*c$, $c\in H^*(\mL;\CC)$. Again we will only compute the case $n=3$. In this case the fibre is
\[\CC\PP^2\cong U(3)/U(2)\times U(1)\]
We will denote by $A$ and $B$ the tautological $U(2)$ and $U(1)$-bundles respectively. The integrals we need to compute are a subset of
\[\ev_!c_2(A)^ac_1(B)^b\]
since the cohomology classes $t$ and $e$ from the previous section pullback to $c_1(B)$ and $c_2(A)$ respectively. The integrals we are interested in are
\begin{align}\label{coh-ev}
\ev_!c_1(B)^k&=\begin{cases}
1&\ \text{when}\ k=2\\
c_1(\mH)&\ \text{when}\ k=3\\
c_1^2(\mH)-c_2(\mH)=c_2(\mH)=\frac{c_1(\mH)^2}{2}&\ \text{when}\ k=4\\
\end{cases}\\
\ev_!c_2(A)c_1(B)^k&=\begin{cases}
1&\ \text{when}\ k=0\\
0&\ \text{when}\ k=1\\
0&\ \text{when}\ k=2\\
c_3(\mH)=\tau^*\chi&\ \text{when}\ k=3\\
c_1(\mH)\cup c_3(\mH)=c_1(\mH)\cup\tau^*\chi&\ \text{when}\ k=4\\
\end{cases}\\
\ev_!c_2(A)^2&=c_2(\mH)=\frac{c_1(\mH)^2}{2}
\end{align}
where $\mH$ is the tautological $U(3)$-bundle over $Z$.

\subsection{The topology of hyperbolic manifolds}\label{vis}
The assumption that the Stiefel-Whitney classes of $M$ vanish allows us to give a good answer to the question of when $M$ admits an almost complex structure. Notice that the inclusion of $SO(2n)$ into $GL^+(2n)$ is a homotopy equivalence and hence the question of whether $M$ admits an orthogonal almost complex structure (a section of the twistor bundle) is the same as whether it admits any almost complex structure.
\begin{lma}
An 6-manifold with vanishing Stiefel-Whitney classes admits: a) an almost complex structure, b) a field $\xi$ of tangent 2-planes.
\end{lma}
\begin{proof}
To prove a), recall (\cite{OkVan}, Proposition 8) that the obstruction to the existence of a lift $M\rightarrow BU(3)$ of the classifying map for the oriented tangent bundle is the Bockstein image of $w_2(M)$ in $H^3(M;\ZZ)$. To prove b), (\cite{Th}, Theorem 1.3) gives a condition for the existence of such a 2-plane field with Euler class $u\in H^2(M;\ZZ)$: that there exists a class $v\in H^4(M;\ZZ)$ which reduces to the 4th Stiefel-Whitney class such that $u\cup v[M]=\chi(M)$. Vanishing of the Stiefel-Whitney classes reduces this to the condition that $v$ is divisible by 2. Since the Euler characteristic of a 6-manifold is even the existence of such a class $v$ follows straight from nondegeneracy of the Poincar\'{e} pairing.
\end{proof}
The almost complex structure gives us a section of the twistor bundle and hence a submanifold representing the homology class Poincar\'{e}-dual to $c_1(\mH)^3/8$. The 2-plane field allows us to ``Gauss-lift'' the whole 6-manifold by defining a submanifold of the twistor bundle consisting of points corresponding to complex structures for which $\xi$ is preserved. This submanifold intersects each twistor fibre in a twistor line and is Poincar\'{e} dual to $c_1(\mH)^2/4$. We can also define a submanifold Poincar\'{e} dual to $c_1(\mH)$ by taking the fibrewise cut-locus of our section (which is a hyperplane in each fibre). Let us write $\Sigma_k$ ($k=0,1,2,3$) for the corresponding submanifold representing $\pd\left(\frac{c_1(\mH)^k}{2^k}\right)$ (where $\Sigma_0=Z$). Notice that if $Y$ is a submanifold of $M$ then $\tau^{-1}(Y)$ intersects $\Sigma_k$ transversely for all $k$.

These observations allow us to visualise homology classes in the twistor space. Given a cohomology class $y$ in $M$ we can represent $K\pd(y)$ by a submanifold $Y$ for large $K$. Now we can represent $K\pd(c_1(\mH)^i\tau^*y)$ by the intersection of the preimage $\tau^{-1}(Y)$ with the submanifolds (sections, Gauss-lifts,...) representing $c_1(\mH)^i$ (assuming they exist).
\section{Definition of Gromov-Witten invariants}\label{dfngw}
We recall the definition of genus 0 Gromov-Witten invariants, quantum cohomology and the Gromov-Witten potential. For more details see \cite{MS04}. Let $Z$ be a $2N$-dimensional monotone symplectic manifold (for example the twistor space of a hyperbolic $2n$-manifold with $n\geq 3$, for which $N=\frac{n(n+1)}{2}$).
\begin{dfn}
Let $J$ be a regular $\omega$-tame almost complex structure on $Z$, $\beta\in H_2(Z;\ZZ)$ a homology class and define
\begin{align*}
\mM_{0,k}(Z,\beta,J)&:=\{(u,\underline{z})|u:S^2\rightarrow X,\ \dbar_Ju=0,\ u_*[S^2]=\beta,\ \underline{z}=(z_1,\ldots,z_k)\in S^k,\\
&\ \ \ \ \ \ \ \ \ \ z_1=0,\ z_2=1,\ z_3=\infty,\ z_i\neq z_j\text{ for }i\neq j\}
\end{align*}
This is a smooth manifold of dimension
\[2N+2c_1(Z)[\beta]+2(k-3)\]
Consider the evaluation map
\[\ev_k:\mM_{0,k}(Z,\beta,J)\rightarrow Z^k,\ \ev_k(u)=(u(z_1),\ldots,u(z_k))\]
This is a pseudocycle which can be compactified by adding strata of stable maps to $\mM_{0,k}(Z,\beta,J)$. The compactified moduli space is denoted
\[\overline{\mM}_{0,k}(Z,\beta,J)\]
We define the \emph{$k$-point genus 0 Gromov-Witten invariant of $Z$ in the class $A$} to be the homology class
\[\GW^Z_{\beta,k}=(\ev_k)_*[\mM_{0,k}(Z,\beta,J)]\in H_{2N+2c_1(Z)+2(k-3)}(Z^k;\CC)\]
in the sense of pseudocycles. One can extract numerical invariants by intersecting with pseudocycles in $Z^k$. If $a_1,\ldots,a_k$ are pseudocycles representing $\CC$-homology classes in $Z^k$ then we write
\[\GW^Z_{\beta,k}(a_1,\ldots,a_k)=\GW^Z_{\beta,k}\cdot(a_1\times\cdots\otimes a_k)\]
where $\cdot$ is the pseudocycle intersection pairing.
\end{dfn}
In the case of twistor spaces all holomorphic curves live in a multiple $mA$ of the homology class of a twistor line and when $n\geq 3$ the twistor space is monotone so we will use the Novikov coefficient ring $\Lambda=\CC[q]$ (the exponent of $q$ corresponds to the multiplicity $m$ and monotonicity implies we only need polynomials rather than formal power series).

Pick a $\ZZ$-basis $x_0,\ldots,x_N$ for $H^*(Z;\ZZ)$ with $x_0=1\in H^0(Z;\ZZ)$ such that every basis element has pure degree.
\begin{dfn}
The quantum cohomology of $(Z,\omega)$ is a ring structure on the graded vector space
\[QH^*(Z;\Lambda)=H^*(Z;\CC)\otimes_{\CC}\Lambda\]
given (on elements of pure degree) by
\[a\star b=\sum_{\beta\in H_2(Z)}\sum_{\nu,\mu}\GW^Z_{\beta,3}(a,b,x_i)g^{ij}x_j\otimes e^{\beta}\]
where $g^{ij}$ is the inverse matrix to
\[g_{ij}=\int_Zx_i\cup x_j\]
and the grading on $QH^*(Z;\Lambda)$ is
\[QH^k(Z;\Lambda)=\bigoplus_i H^i(Z;\CC)\otimes_{\CC}\Lambda^{k-i}\]
\end{dfn}
\section{The linearised theory}\label{linear}
Recall that vertical $J_-$-holomorphic curves are $j_F$-antiholomorphic curves in the twistor fibre. The image of an antiholomorphic curve is equal to the image of a holomorphic curve since we can always precompose with an antiholomorphic involution. We will now prove the following.
\begin{prp}\label{lintheory}
For vertical $J_-$-holomorphic curves,
\begin{enumerate}[(i)]
\item the kernel of the linearised $\dbar_{J_-}$-operator is precisely the tangent space of the moduli space.
\item If moreover the image of a $J_-$-holomorphic curve is contained in a line
\[SO(4)\times U(n-2)/U(2)\times U(n-2)\]
in the twistor fibre corresponding to a 4-plane $\pi\subset T_pM$ then there is a subspace of the cokernel of the linearised $\dbar_{J_-}$-operator naturally isomorphic to $\pi$.
\item In particular, the \emph{obstruction bundle} of cokernels over the moduli space
\[\mL=\Gamma\backslash SO^+(2n,1)/SO(4)\times U(n-2)\]
of (anti)twistor lines is naturally isomorphic to the tautological $SO(4)$-bundle.
\end{enumerate}
\end{prp}
One might be concerned that using antiholomorphic curves in the twistor fibre will affect the orientation of the moduli space. Indeed the orientation of the moduli space of twistor lines is reversed along the fibre directions, but so is the orientation of the twistor fibre and therefore the computation of Gromov-Witten invariants will not be affected if we ignore this simultaneous change of signs.

Let $u:\CC\PP^1\rightarrow Z$ be a genus 0 vertical $J_-$-holomorphic curve.
\begin{itemize}
\item Let $u^*\nabla$ denote the pullback of the twistor connection to $u^*TZ$.
\item Recall that the linearised Cauchy-Riemann operator $D_u:\Omega^0(u^*TZ)\rightarrow\Omega^{0,1}(u^*TZ)$ is given by the formula
\[(D_u\xi)(X)=(u^*\nabla)_X\xi+J_-(u^*\nabla)_{j X}\xi-(J_-\nabla_{\xi}J_-)du(X)\]
\end{itemize}
We can take vertical and horizontal parts of $D_u$ using Lemma \ref{formulae} and we get
\[D_u=\left(
\begin{array}{cc}
D_u^{HH} & 0\\
D_u^{VH} & D_u^{VV}
\end{array}\right):\Omega^0(u^*\mH\oplus u^*\mV)\rightarrow\Omega^{0,1}(u^*\mH\oplus u^*\mV)
\]
where $D_u^{VH}\xi=(D_u\xi^H)^V$ etc. The $D_u^{VV}$-part is just the linearised Cauchy-Riemann operator governing deformations of $u$ as a holomorphic curve in $F$ (the twistor fibre). But $F$ is a homogeneous space and is therefore convex in the sense of Kontsevich \cite{Kont}, i.e. all genus 0 holomorphic curves in $F$ are regular. Therefore $D_u^{VV}$ is surjective.

The kernel of $D_u$ fits into an exact sequence
\[0\rightarrow\ker D_u^{VV}\stackrel{a}{\rightarrow}\ker D_u\stackrel{b}{\rightarrow}\ker D_u^{HH}\rightarrow 0\]
where $a$ is inclusion into the vertical component and $b$ is projection to the horizontal component: surjectivity of $b$ follows from surjectivity of $D_u^{VV}$. The cokernel of $D_u$ is naturally identified with the cokernel of $D_u^{HH}$, therefore to compute obstruction bundles it suffices to understand $D_u^{HH}$.

Using the projection $\tau_*$ and horizontal lifting $\tilde{\cdot}$ we will identify $\mH$ and $\tau^*TM$. Let $e_1,\ldots,e_{2n}$ be a local orthonormal frame on $M$ near a point $p$ and $\{\tilde{e}_i\}_{i=1}^{2n}$ the horizontal lift of this frame to a neighbourhood of $\tau^{-1}(p)$ in $Z$. Suppose $\xi=\sum\xi^i\tilde{e}_i\in\mH$. Then
\begin{align*}
(D_u^{HH}\xi)(X)&=[D_u\xi(X)]^H\\
&=X(\xi^i)\tilde{e}_i+\xi^i\nabla_X\tilde{e}_i+(jX)(\xi^i)\psi\tilde{e}_i+\xi^i\psi\nabla_{jX}\tilde{e}_i\\
&\ \ \ \ \ -\xi^i\psi\left(\nabla_{\tilde{e}_i}du(jX)-J_-\nabla_{\tilde{e}_i}X\right)^H
\end{align*}
where $\psi$ is the complex structure on $\mH$ at the point $u(z)$ of the twistor fibre. By Lemma \ref{formulae} this equation is just
\begin{equation}\label{CR}\boxed{D_u^{HH}\xi(X)=X(\xi^i)\tilde{e}_i+(jX)(\xi^i)\psi\tilde{e}_i}\end{equation}
Pick conformal coordinates $(s,t)$ in a patch on $\CC\PP^1$ and write $X=X^s\partial_s+X^t\partial_t$. Then Equation (\ref{CR}) becomes
\[D_u^{HH}\xi(X)=(X^s\delta^i_{\ j}-X^t\psi^i_{\ j})(\partial_s\xi^j+\psi^j_{\ k}\partial_t\xi^k)\]
\begin{proof}[Proof of Proposition \ref{lintheory} (i)]
To find the kernel of $D_u^{HH}$ it therefore suffices to solve
\[\Xi:=\partial_s\xi^j+\psi^j_{\ k}\partial_t\xi^k=0\]
Differentiating this with respect to $s$ gives
\[\partial_s\Xi^j=\partial_s^2\xi^j+(\partial_s\psi^j_{\ k})(\partial_t\xi^k)+\psi^j_{\ k}\partial_s\partial_t\xi^k=0\]
and with respect to $t$ gives
\[\partial_t\Xi^{\ell}=\partial_t\partial_s\xi^{\ell}+(\partial_t\psi^{\ell}_{\ k})(\partial_t\xi^k)+\psi^{\ell}_{\ k}\partial_t^2\xi^k=0\]
so
\[\partial_s\Xi^j-\psi^j_{\ \ell}\partial_t\Xi^{\ell}=\nabla^2\xi^j+(\partial_s\psi^j_{\ k}-\psi^j_{\ \ell}\partial_t\psi^{\ell}_{\ k})\partial_t\xi^k\]
since $\nabla^2=\partial_s^2+\partial_t^2$ in local coordinates. Note that by Equation (\ref{psieqn})
\[\partial_s\psi^j_{\ k}-\psi^j_{\ \ell}\partial_t\psi^{\ell}_{\ k}=0\]
so $\xi^k\in\ker D_u^{HH}$ implies that the component functions $\xi^j$ are harmonic functions on $S^2$ and hence constant.

Therefore the kernel of $D_u$ consists of precisely the tangent directions in the moduli space $\mM(mA,J_-)$ (where $u$ has degree $m$) since these are precisely the deformations of $u$ as a vertical curve (kernel of $D_u^{VV}$) and the deformations of $u$ to nearby fibres (kernel of $D_u^{HH}$).\end{proof}
\begin{proof}[Proof of Proposition \ref{lintheory} (ii)]
\begin{lma}
The adjoint operator $(D_u^{HH})^*$ is given (in conformal coordinates on the domain and coordinates on $u^*\mH$ induced from an orthonormal frame $e_i$ of $T_pM$ as before) by
\[((D_u^{HH})^*\eta)^i=-2\Theta^{-2}(\partial_s\eta^i-\partial_t(\psi^i_{\ j}\eta^j))\]
where $\dvol=\Theta^2 ds\wedge dt$ is the standard area form on $S^2$ expressed in conformal coordinates and $\eta^i=\eta(\partial_s)^i$ (since $\eta\in\Omega^{0,1}(u^*\mH)$ its value on any vector determines its value on any other).
\end{lma}
We can now write down some solutions of $(D_u^{HH})^*\eta=0$ immediately.
\begin{lma}
Let $v\in T_pM$ be a vector and pick conformal coordinates $(s,t)$ on a patch in $S^2$. Define $\eta_v=\Omega^{0,1}(u^*\mH)$ by requiring $\eta_v(\partial_s)^i=(\partial_t\psi^i_{\ j})v^j$. Then $(D_u^{HH})^*\eta_v=0$.
\end{lma}
\begin{proof}
Since $v$ is constant we have
\[((D_u^{HH})^*\eta_v)^i=-\frac{2}{\Theta^2}\left(\partial_s\partial_t\psi^i_{\ k}-\partial_t(\psi^i_{\ j}\partial_t\psi^j_{\ k})\right)v^k\]
which vanishes by Equation (\ref{psieqn}).\end{proof}
Note that if $\psi$ lands in a line in $F$ corresponding to the plane $\pi$ then $\partial_t\psi^i_{\ j}v^j\in\pi$. Since $u$ is somewhere immersed the correspondence $v\mapsto\eta_v$ is an isomorphism for $v\in\pi$ between $\pi$ and $\ker(D^{HH}_u)^*\cong\coker D_u^{HH}$ as claimed.\end{proof}

\subsection{Interpretation}\label{pert-slice}
Let $u$ be a $J_-$-holomorphic curve in $Z$. Observe that an infinitesimal deformation of $\omega$-compatible almost complex structures $J_-\mapsto J_-+\delta J$ gives us a natural choice of $\eta\in\Omega^{0,1}(u^*TZ)$ defined by
\[\eta(X)=\delta J(u_*(X))\]
We have defined elements
\[\eta_v:\partial_s\mapsto(\partial_t\psi)v,\ \partial_t\mapsto-(\partial_s\psi)v\]
in the cokernel of $D_u$. Here $\partial_t\psi=u_*(\partial_t)$, $\partial_s\psi=u_*(\partial_s)$. We can now understand these as coming from the following infinitesimal deformation of $J_-$
\[
\delta_v J(w)=\begin{cases}
(J_-w)v &\text{ when } w\in \mV\\
0&\text{ when } w\in\mH
\end{cases}
\]
\section{The Gromov-Witten theory of twistor lines}\label{gw-line}
\subsection{The obstruction bundle}
In this section we compute the Gromov-Witten cycles associated to the moduli space of pseudohololmorphic spheres in the homology class $A$. We recall that this moduli space has a very nice description as a homogeneous space
\begin{align*}
\mL&=\mM_{0,0}(Z,A,J_-)=\Gamma\backslash SO^+(2n,1)/SO(4)\times U(n-2)\\
\mL_1&=\mM_{0,1}(Z,A,J_-)=\Gamma\backslash SO^+(2n,1)/U(2)\times U(n-2)\\
\end{align*}
The key observation is that the moduli space is compact ($A$ is a minimal homology class). We can therefore apply the following theorem:
\begin{thm}[See \cite{MS04}, Proposition 7.2.3]
Let $Z$ be a semipositive symplectic manifold and $A$ a homology class which is not a multiple cover of a homology class $B$ with $c_1(Z)[B]=0$. If the moduli space $\mM_{0,0}(Z,A,J)$ is compact and smooth with tangent space at $u$ equal to the kernel of the linearised $\dbar$-operator $D_u$ then the cokernels of $D_u$ form a smooth vector bundle over $\mM_{0,0}(Z,A,J)$ called the obstruction bundle $\Obs$ and the Gromov-Witten class $\GW^Z_{A,k}$ may be computed by
\[\pd(\ev_k)_!\ft_k^*\obs\]
where
\begin{itemize}
\item $\pd:H^*\rightarrow H_*$ denotes Poincar\'{e} duality, $_!$ denotes cohomological pushforward,
\item $\ft_k:\mM_{0,k}(Z,A,J_-)\rightarrow\mM_{0,0}(Z,A,J_-)$ is the forgetful map,
\item $\ev_k:\mM_{0,k}(Z,A,J_-)\rightarrow Z^k$ is the evaluation map and
\item $\obs$ is the Euler class of the obstruction bundle over $\mM_{0,0}(Z,A,J_-)$.
\end{itemize}
\end{thm}
In Section \ref{linear} we proved that we are in precisely this setting and that
\begin{thm}\label{obstructionthm}
The obstruction bundle over $\mL$ is isomorphic to the tautological $SO(4)$-bundle.
\end{thm}
\begin{rmk}\label{regularisedmoduliremark}
One may think of the zero-set of a section of the obstruction bundle as a ``regularised moduli space'' $\widetilde{\mM}_{0,k}(Z,A,J_-)$. There is then a very appealing picture. Take a vector field $V$ on $M$ and lift it to a section of $\Obs$ by projecting it to the tautological 4-plane bundle over the moduli space $\mL$. The regularised moduli space consists of twistor lines corresponding to 4-planes in $T_pM$ which are orthogonal to $V(p)$. If $\widetilde{\mM}_{0,k}(p)$ denotes that part of the regularised moduli space living in the twistor fibre at $p$ then whenever $V(p)\neq 0$,
\[\widetilde{\mM}_{0,0}(p)\cong SO(2n-1)/SO(4)\times U(n-3)\]
where $SO(2n-1)$ is the stabiliser of $V(p)$, and
\[\widetilde{\mM}_{0,1}(p)\cong SO(2n-1)/U(2)\times U(n-3)\stackrel{\ft}{\rightarrow} SO(2n-1)/SO(4)\times U(n-3)=\widetilde{\mM}_{0,0}(p)\]
tells us, morally, which curves will persist in that fibre after deformation of $J$. In the case $n=3$ this reduces to the standard fibration
\[SO(5)/U(2)=\CC\PP^3\rightarrow S^4=SO(5)/SO(4)\]
so we see (at least heuristically) that $Z$ is uniruled if $\dim(M)=6$.
\end{rmk}
\subsection{The algorithm}
We now give an algorithm which can be used to compute
\[\GW^Z_{A,k}=\pd(\ev_k)_!\ft_k^*\obs\]
The map $\ev_k$ factors as
\[\mM_{0,k}(Z,A,J_-)\stackrel{\rem_1\times\cdots\times\rem_k}{\xrightarrow{\hspace*{1.7cm}}}\mL_1\times\cdots\times\mL_1\stackrel{\ev\times\cdots\times\ev}{\xrightarrow{\hspace*{1.2cm}}} Z^k\]
where $\rem_j$ is the map remembering only the $j$-th marked point and
\[\ev:\mL_1=\Gamma\backslash SO^+(2n,1)/U(2)\times U(n-2)\rightarrow\Gamma\backslash SO^+(2n,1)/U(n)=Z\]
is the 1-point evaluation map. This factorisation fits into a diagram
\[
\begin{CD}
\mM_{0,k}(Z,A,J_-) @>{\prod_{j=1}^k\rem_j}>> \mL_1^k @>{(\ev)^k}>> Z^k\\
@V{\ft_k}VV @VV{(\ft_1)^k}V @.\\
\mL @>>{\Delta}> \mL^k @.
\end{CD}
\]
which implies that we need to compute
\begin{equation}\label{gwfirststep}\pd(\ev)^k_!((\ft_1)^k)^*\Delta_!\obs\end{equation}
We have $\Delta^k_!\obs=(\obs\otimes 1^{\otimes k-1})\cup\Delta^k_!1$. Recall that the projection
\[\lambda:\mL\rightarrow M\]
is a $SO(2n)/SO(4)\times U(n-2)$-bundle and that, by the Leray-Hirsch theorem, $H^*(\mL;\CC)$ is a free $H^*(M;\CC)$-module with some collection of generators $\{z_i\}$ arising as characteristic classes of the tautological $SO(4)$ and $U(n-2)$-bundles over the moduli space. Let $\{y_j\}$ be a basis for $H^*(M;\CC)$ and $z_i\lambda^*y_j$ the corresponding basis for $H^*(\mL;\CC)$.  We take a decomposition of the diagonal (Section \ref{diagdecomp})
\[\Delta^k_!1=\sum_{i_1,\ldots,i_k,j_1,\ldots,j_k}P_{\mL}^{i_1j_1,\ldots,i_kj_k}(z_{i_1}\lambda^*y_{j_1})\otimes\cdots\otimes (z_{i_k}\lambda^*y_{j_k})\]
Now substituting in Equation \eqref{gwfirststep} and using the fact that
\[\lambda\circ\ft_1=\tau\circ\ev\]
we get
\begin{thm}\label{GWbigformula-thm}
\begin{align}\label{GWbigformula}\nonumber
\GW^Z_{A,k}&=\pd\sum P_{\mL}^{i_1j_1,\ldots,i_kj_k}((\ev_!(\ft_1^*(z_{i_1}\cup\obs)))\cup\tau^*y_{j_1})\otimes\\
&\ \ \ \ \ \ \ \ \ \ \otimes((\ev_!\ft_1^*z_{i_2})\cup\tau^*y_{j_2})\otimes\cdots\otimes((\ev_!\ft_1^*z_{i_k})\cup\tau^*y_{j_k})\end{align}
\end{thm}
It remains only to find the decomposition of the diagonal and to compute the fibre integrals
\[\ev_!\ft_1^*z_m,\ \ev_!(\ft_1^*(z_m\cup\obs))\]
along the $U(n)/U(2)\times U(n-2)$-bundle $\ev$ which can be done using Borel-Hirzebruch theory as in Section \ref{ev-pushfwd}.
\subsection{Examples}
We will illustrate the use of this algorithm through a number of elementary examples.
\begin{cor}
Suppose $n=2$ (so $M$ is a hyperbolic 4-manifold). Then
\[\GW^Z_{A,k}=\chi(M)A^{\otimes k}\in H^*(Z^k;\CC)\]
where $A$ is the homology class of the twistor fibre.
\end{cor}
\begin{proof}
In this dimension $\ev:\mL_1\rightarrow Z$ and $\lambda:\mL\rightarrow M$ are diffeomorphisms so $\{z_m\}=\{1\}$ and \eqref{GWbigformula} reduces to
\[\pd\sum P_M^{j_1\ldots j_k}(\ft_1^*\obs\cup\tau^*y_{j_1})\otimes\tau^*y_{j_2}\otimes\cdots\otimes\tau^*y_{j_k}\]
where (by Theorem \ref{obstructionthm}) $\obs$ is the Euler class of the tangent bundle of $M$ and $\ft_1=\tau$ so $\tau^*\obs=c_2(\mH)=\pd(A)$. Therefore a summand is nonzero if and only if $y_{j_1}$ has nontrivial cup product with $\chi$. This means that $y_{j_1}$ must have degree zero and that therefore all other $y_{j_m}$ must be of top degree in order that the Poincar\'{e} amplitude $P^{j_1\ldots j_k}_M$ is nonvanishing. Under Poincar\'{e} duality these pullback (via $\tau^!$) to $A$, the homology class of the twistor fibre, and we get
\[\GW^Z_{A,k}=A^{\otimes k}\in H^*(Z^k;\CC)\]\end{proof}
\begin{cor}\label{gw-input}
Suppose $n=3$ (so $M$ is a hyperbolic 6-manifold).
\begin{enumerate}
\item  When $k=1$,
\[\GW^Z_{A,1}=[Z]\]
i.e. $Z$ is uniruled.
\item When $k=2$,
\begin{align*}
\GW^Z_{A,2}&=\frac{1}{4}\pd\left(\{1\otimes c_1(\mH)^2\}\cup\tau^*\Delta_!^2(1_M)\right)
\end{align*}
where $\mH$ is the horizontal distribution on $Z$ considered as a complex vector bundle.
\item When $k=3$,
\begin{align*}
\GW^Z_{A,3}&=\pd\left(\frac{1}{16}\{1\otimes c_1(\mH)^2\otimes c_1(\mH)^2\}\cup\tau^*\Delta^3_!(1_M)+\right.\\
&\left.\ \ \ +\frac{\chi(M)}{2}\{c_1(\mH)\otimes 1\otimes 1\}\cup\tau^*\Delta_!^3(\vol_M)\right)
\end{align*}
\end{enumerate}
where we have freely used the notation of Corollary \ref{cor-diagdecomp-Z}.
\end{cor}
\begin{proof}[Proof of $k=1$:] In view of \eqref{GWbigformula} it suffices to compute $\ev_!c_2(\mH)$ where $\mH$ is the tautological $U(2)$-bundle over $\mL_1$. We saw in Section \ref{ev-pushfwd} that
\[\ev_!c_2(\mH)=1=\pd^{-1}[Z]\]
\end{proof}
\begin{proof}[Proof of $k=2$:] By Corollary \ref{cor-diagdecomp}, we have the following decomposition of the diagonal
\begin{align*}
\Delta^2_!1&=\frac{1}{2}\sum_{i,j}g_M^{ab}\left(\lambda^*y_a\otimes t^4\lambda^*y_b+t\lambda^*y_a\otimes t^3\lambda^*y_b+t^2\lambda^*y_a\otimes t^2\lambda^*y_b\right.\\
&\left.\ \ \ \ +t^3\lambda^*y_a\otimes t\lambda^*y_b+t^4\lambda^*y_a\otimes\lambda^*y_b+e\lambda^*y_a\otimes e\lambda^*y_b\right)
\end{align*}
Now the result follows from \eqref{GWbigformula} and the formulae for cohomological pushforward along the evaluation map \eqref{coh-ev}.
\end{proof}
\begin{proof}[Proof of $k=3$:] Follows similarly.\end{proof}

\begin{rmk}
We can understand the case $k=2$ heuristically as follows. First, notice that the class $\frac{1}{4}c_1(\mH)^2$ pulls back to the twistor fibre as the cohomology class of a twistor line (since $c_1(\mH)$ pulls back to $2H$). Now consider a homology class $\mu\in H_k(M;\CC)$ represented by an oriented submanifold $N\subset M$. As we saw in Section \ref{vis} we can always find an almost complex structure $\psi$ on $M$ and (if the homology is torsionfree) a 2-plane field $\xi$. Define the submanifolds $N_1',\ N_1'',\ N_1'''$ of $Z$ which fibre over $N_1$ with fibre $F_1',\ F_1'',\ F_1'''$ over $n\in N_1$ equal to
\begin{itemize}
\item $F_1''$: the set of complex structures on $T_nZ$ making $\xi$ holomorphic (like a Gauss lift, consisting of a line in the twistor fibre),
\item $F_1'$: the point $\psi(n)$,
\item $F_1'''$: the cut locus of $\psi(n)$ (a copy of $\CC\PP^2$ in the twistor fibre).
\end{itemize}
Suppose that the Poincar\'{e} dual class $\check{\mu}$ can also be represented by a submanifold $N_2$, which intersects $N_1$ exactly once transversely at some point $n$. Then the Gromov-Witten contribution from twistor lines to their quantum intersection is
\[1=\GW^Z_{A,2}([N'_1],[N'''_2])=\GW^Z_{A,2}([N'''_1],[N'_2]),\ \GW^Z_{A,2}([N''_1],[N''_2])=0\]
from our formula since $[N'_1]=H\cup\tau^*\mu$, $[N'''_2]=H^3\cup\tau^*\check{\mu}$, etc. This can be seen via our earlier heuristic picture of the regularised moduli space (Remark \ref{regularisedmoduliremark}) by noticing that when the perturbing vector field $V$ is chosen with $V(n)\neq 0$ there is a unique twistor line in the regularised moduli space connecting the point $F'_1$ with the cycle $F'''_2$ and that the spheres $F''_1$ and $F''_2$ will generically project to non-intersecting spheres in $S^4=\widetilde{\mM}_{0,0}(n)$ and hence there are no connecting twistor lines.

One must be careful with this heuristic because it can be misleading. At first sight, if one had a pair of sections of the twistor bundle then their quantum intersection would pick up the $\chi(M)$ lines joining them inside the unperturbed fibres (where there is still a line through every pair of points) but a simple dimension count shows this is not the case: the Gromov-Witten class has dimension 14 while the pair of sections would have codimension 12 in $Z^2$ and we see that the chosen regularisation of the moduli space is not transverse to such a submanifold.
\end{rmk}

\section{Higher degree curves}\label{high-degree}
\subsection{Easy computations}
Some higher degree contributions are easy to compute for dimension reasons. Recall that the Gromov-Witten invariant lives in degree
\[\deg\GW^Z_{mA,k}=n(n+1)+4m(n-2)+2(k-3)\]
and $\dim Z^k=kn(n+1)$. This gives us the trivial bound
\[4m(n-2)+2(k-3)\leq (k-1)n(n+1)\]
necessary for the nonvanishing of the invariant. For example,
\begin{cor}
When $n=3$,
\begin{itemize}
\item the 1-point invariant only gets contributions from curves of degree 1,
\item the 2-point invariant only gets contributions from curves of degree 3 or less,
\item the 3-point invariant only gets contributions from curves of degree 6 or less.
\end{itemize}
\end{cor}
However the special geometry of the Eells-Salamon almost complex structure gives us more information still.
\begin{lma}\label{triv-van}
Suppose $n=3$ and $\{y_i\}_{i=1}^k$ are cohomology classes on $M$ with degrees $d_i$. If $\sum_{i=1}^kd_i>6$ then
\[\GW_{mA,k}(c_1^{i_1}\tau^*y_{1}\otimes\cdots\otimes c_1^{i_k}\tau^*y_{k})=0\]
\end{lma}
\begin{proof}
Let us assume without loss of generality that the homology classes $\pd(y_i)$ are represented by submanifolds $Y_i$ of $M$ (we can rescale by a large integer). Recall from Section \ref{vis} that we have submanifolds $\Sigma_p$ representing $\pd(c_1(\mH)^p)$ for $p=0,1,2,3$ which transversely intersect the preimages $\tau^{-1}(Y_i)$so the Gromov-Witten invariant
\[\GW_{mA,k}(c_1^{i_1}\tau^*y_{1}\otimes\cdots\otimes c_1^{i_k}\tau^*y_{k})\]
counts (for a generic $J$) the number of $J$-holomorphic curves passing through all of the $\tau^{-1}(Y_{\ell})\cap\Sigma^{(\ell)}_{i_{\ell}}$ (where $\Sigma^{(\ell)}_{i_{\ell}}$ is choice of section/Gauss-lift/cut-locus, not necessarily the same for each value of $i_{\ell}$ for the sake of transversality).

If $\sum_{i=1}^kd_i>6$ then $\sum_{\ell=1}^k\deg(Y_{\ell})< 6(k-1)$ and we can perturb the submanifolds $Y_{\ell}$ so that the intersection $\bigcap_{\ell=1}^kY_{\ell}$ is empty. Then the moduli space of $J_-$-holomorphic curves which touch all of $\Sigma^{(\ell)}\cap \tau^{-1}(Y_{\ell})$ is empty since all $J_-$-holomorphic curves are vertical. Therefore the Gromov-Witten invariant is zero. 
\end{proof}
\begin{cor}\label{triv-van-cor}
Suppose $n=3$. We have $\GW_{mA,k}=0$ if $k<m$.
\end{cor}
\begin{proof}
For degree reasons we know that
\[12+4m+2(k-3)+\sum_{\ell=1}^k(6+\deg(Y_{\ell})-2i_{\ell})=12k\]
which gives
\[\sum_{\ell=1}^k\deg(Y_{\ell})=4(k-m)-6+2\sum_{\ell=1}^ki_{\ell}\]
Since $i_{\ell}\leq 3$ we get
\[\sum_{\ell=1}^k\deg(Y_{\ell})\leq 10k-4m-6\]
The inequality $k<m$ ensures that $10k-4m-6<6(k-1)$.
\end{proof}
\subsection{Obstruction method}
Now we use an `obstruction bundle' argument to deal with the case $k=m$.
\begin{thm}\label{crit-thresh}
Let $M$ be a hyperbolic 6-manifold ($n=3$) and $m\geq 3$ an integer. We have
\begin{align}
\label{GWa}\GW_{mA,m}&=0\\
\label{GWb}\GW_{2A,3}&=0
\end{align}
It then follows from the divisor equation that
\[\GW_{2A,2}=0.\]
\end{thm}
The proof will proceed by observing that in these cases there is a nonvanishing section of the obstruction bundle over the whole moduli space. If one is willing to appeal to a general theory of Kuranishi structures \`{a} la Fukaya-Ono \cite{FO} that is enough to prove vanishing of the Gromov-Witten invariant, but in our setting it should suffice to perturb the almost complex structure and indeed the explicit sections of the obstruction bundle we have arise (infinitesimally) from precisely such perturbations. Therefore we outline the proof of a slightly more general theorem from which Theorem \ref{crit-thresh} will follow below.

Before we state the theorem, recall that if $u=v_1\cup\cdots\cup v_k$ is a stable curve then the linearised operator $D_u$ is just the restriction of $\bigoplus_{i=1}^kD_{v_i}$ to the subspace of $\bigoplus_{i=1}^kW^{1,p}(v_i^*TZ)$ consisting of $k$-tuples of vector fields which agree at the nodal points. In our setting the image of $D_u$ is precisely the image of $\bigoplus_{i=1}^kD_{v_i}$. To see this, observe that if $\eta=D_{v_i}\xi_i\in\im(D_{v_i})$ then there is a vector field $\xi_j\in\ker(D_{v_j})$ for all $j$ such that $(\xi_1,\ldots,\xi_k)$ agree at the nodes: when the domain of $v_j$ has a node $n$ connecting it with the domain of $v_{j'}$ for which $\xi_j$ has already been constructed we let $\xi_{j'}^V$ be a vertical vector field in the kernel of $D^{VV}_{v_j'}$ which agrees with $\xi_j^V$ at $n$ (which exists by the transitive isometric action of $SO(2n)$ on the twistor fibre) and $\xi^H_{j'}$ be the constant horizontal lift of $\tau_*\xi_j(n)$. In summary:
\begin{lma}
The dimension of $\coker(D_u)$ for any $J_-$-holomorphic stable curve $u$ in the twistor space of a hyperbolic $2n$-manifold depends only on the homology class it represents.
\end{lma}
\begin{thm}\label{GWvanish}
Let
\begin{itemize}
\item $(Z,\omega)$ be a semipositive $2N$-dimensional symplectic manifold and $\mJ$ be the space of $\omega$-compatible domain-dependent almost complex structures (where the domain is $S^2$),
\item $J_-\in\mJ$ be a particular choice of such an almost complex structure,
\item $\beta$ be a homology class in $H_2(Z;\ZZ)$,
\item $X_1,\ldots,X_k$ be a collection of submanifolds such that
\[\codim(X_1\times\cdots\times X_k\subset Z^k)=2N+2c_1(\beta)+2(k-3)\]
and for any $J\in\mJ$
\[\mM(J):=\mM(Z,\beta,\{X_i\},J)\]
denote the moduli space of $J$-holomorphic curves $u$ representing $\beta$ with $k$ marked points $z_1,\ldots,z_k$ in the domain such that $u(z_i)\in X_i$,
\item $\mMbar(J)$ denote the stable map compactification of $\mM(J)$ and $\mM_T(J)$ the stratum of stable maps modelled on a bubble tree $T$,
\item $\exc$ be an even integer (the \emph{excess dimension}),
\end{itemize}
such that
\begin{itemize}
\item each stratum $\mM_T(J_-)$ (modelled on a bubble tree $T$ with $e$ edges) is a smooth manifold of dimension $\exc-2e$ whose tangent space at $u$ is isomorphic to the kernel of the homomorphism
\[\ker(D_u)\rightarrow \oplus_{i=1}^k\nu_{u(z_i)}X_i\]
given by projecting a vector field onto the normal direction to the submanifold $X_i$.
\item the dimension of $\coker(D_u)$ is $\exc$ for all $u\in\mMbar(J)$,
\item for each $u\in\mM_T(J_-)$ with $||du||_{L^{\infty}_0}<c$, each $T'<T$ and each sufficiently small gluing datum $\underline{a}$ there is a neighbourhood $\nu$ of $u$ in the space of stable maps modelled on $T$ and a gluing map
\[\Gl(u,\underline{a},c):\nu\cap\mM_T(J_-,c)\rightarrow\mM_{T'}(J_-,\infty)\]
satisfying Property ($\dag$) of Proposition \ref{glue}.
\item there exists a $\delta J\in T_{J_-}\mJ$ such that
\[\delta J\circ du\circ j\not\in\im(D_u)\]
for all $J_-$-holomorphic stable maps $u\in\mMbar(J)$.
\end{itemize}
Then the genus 0 Gromov-Witten invariant
\[\GW^Z_{\beta,k}(X_1,\ldots,X_k)=0.\]
\end{thm}
When we have defined Property ($\dag$) of Proposition \ref{glue} we will show it is satisfied in our case, see Remark \ref{SalEtAl}. We postpone the proof of Theorem \ref{GWvanish} to Section \ref{prfbigbadthm}.
\begin{proof}[Proof of Theorem \ref{crit-thresh}]
Equation \eqref{GWa} concerns the equality case $m=k$ from the proof of Corollary \ref{triv-van-cor}. Therefore the only nonvanishing Gromov-Witten invariants are of the form
\[\GW_{mA,m}(c_1^3\tau^*y_1\otimes\cdots\otimes c_1^3\tau^*y_m)\]
where $\pd(y_1),\ldots,\pd(y_m)$ are represented by submanifolds $Y_1,\ldots,Y_m$ which intersect transversely in a collection of points $S$. The moduli space of $J_-$-holomorphic curves connecting the submanifolds is now $\coprod_{s\in S}C_s$ where $C_s$ is the space of degree $m$ stable curves in $\tau^{-1}(s)$ passing through the points $p_{\ell}=\Sigma^{(\ell)}_3\cap\tau^{-1}(s)$. Let us write $X_i=\Sigma^{(i)}_3\cap\tau^{-1}(Y_i)$ and let $\{z_i\}_{i=1}^m$ be a collection of distinct points in $S^2$.

Let $v$ be a vector field on $M$ such that at every point $s\in S$, $v$ projects orthogonally to a nonzero vector in any 4-plane corresponding to a twistor line connecting two of the points $p_i$ above $s$. Now $\delta_v J$ as constructed in Section \ref{pert-slice} satisfies the assumptions of Theorem \ref{GWvanish}.

The same argument works for Equation \eqref{GWb} but one must be slightly careful because now the submanifolds $Y_i$ can intersect in something bigger than a point. The only issue is to find a suitable vector field $v$ which has the relevant behaviour over $Y_1\cap Y_2\cap Y_3$, but because this triple intersection is not the whole of $M$ (for dimension reasons) it is always possible to do so.
\end{proof}
\section{Proof of Theorem \ref{GWvanish}}\label{prfbigbadthm}
We begin by stating the relevant implicit function and gluing theorems we need for the proof. We have made our statements as close as possible to those in \cite{MS04} for the reader's convenience. In the sequel $(Z,\omega)$ will always denote a compact symplectic manifold, $\mJ$ the space of $C^r$-differentiable domain-dependent $\omega$-compatible almost complex structures (for some $r>2$).
\begin{dfn}
By an {\em $\epsilon$-perturbation at $J$} we mean a smooth embedding $\kappa:B_{\epsilon}\rightarrow\mJ$ of a finite-dimensional compact Euclidean $\epsilon$-ball centred at $\kappa(0)=J$. For $Y\in B_{\epsilon}$ we will write $J_Y:=\kappa(Y)$, $g_Y$ for the associated almost K\"{a}hler metric and $W^{1,p}_Y,\ L^p_Y,\ C^r_Y$ for norms taken with respect to the metric $g_Y$. We write $K:=T_J\kappa(B_{\epsilon})$. If $u$ is a $W^{1,p}$-map $\Sigma\rightarrow Z$ from a Riemann surface $(\Sigma,j)$ then we denote by $\iota_u:K\rightarrow \Omega^{0,1}(u^*TZ)$ the map sending $Y\in K$ to $\frac{1}{2}Y\circ du\circ j$ (we blur the distinction between $K$ and $B_{\epsilon}$, writing $Y$ for elements of either). We will also write $\mB_Y$ for the Banach manifold of $W^{1,p}_Y$-maps from $\Sigma$ to $Z$ representing some homology class $\beta$.
\end{dfn}
Recall that if $d\vol$ is a volume form on a complex Riemann surface $(\Sigma,j)$ then for any $p>0$ we denote by $c_p(d\vol_{\Sigma})$ the norm of the Sobolev embedding $W^{1,p}(\Sigma)\hookrightarrow C^0(\Sigma)$ where the norm on $C^0(\Sigma)$ is the $L^{\infty}$-norm. We also note that for any Riemannian vector bundle $E\rightarrow\Sigma$ the $L^{\infty}$-norm of a section is bounded above by $c_p(d\vol_{\Sigma})$ times its $W^{1,p}$-norm (see \cite{MS04}, Remark 3.5.1).
\begin{prp}[Implicit function theorem]\label{IFT}
Let $(\Sigma,j)$ be a compact Riemann surface and $p>2$. Let $\kappa$ be an $\epsilon$-perturbation at $J_0$. Then for every constant $c_0>0$ there exists a constant $\delta>0$ such that the following holds for every volume form $d\vol_{\Sigma}$ on $\Sigma$ satisfying $c_p(d\vol)\leq c_0$. Suppose $u\in W^{1,p}_0(\Sigma,Z)$ and $(\xi_0,Y_0)\in W^{1,p}_0(\Sigma,u^*TZ)\times T_{J_0}\kappa(B_{\epsilon})$ satisfy
\[||du||_{L^p_0}\leq c_0,\ ||\xi||_{W^{1,p}_0}\leq\frac{\delta}{16},\ ||Y||_{C_0^r}\leq\frac{\delta}{16},\ ||\dbar_{J_Y}(\exp^{g_Y}_u(\xi))||_{L^p}\leq\frac{\delta}{4c_0}\]
Moreover suppose that $Q_u:L^p_0(\Sigma,\Lambda^{0,1}\otimes_J u^*TZ)\rightarrow W^{1,p}_0(\Sigma,u^*TZ)\times T_{J_0}\kappa(B_{\epsilon})$ is a right inverse of $D_u+\iota_u$ such that
\[(D_u+\iota_u)Q_u=\id,\ ||Q_u||\leq c_0\]
and suppose that $\iota_u(K)$ is a complementary subspace to $\im(D_u)$. Then there exists a unique $(\xi',Y')=Q_u\eta\in W^{1,p}_0(\Sigma,u^*TZ)$ such that
\[\dbar_{J_{Y+Y'}}(\exp^{g_{Y+Y'}}_u(\xi+\xi'))=0,\ ||\xi+\xi'||_{W^{1,p}_0}\leq\frac{\delta}{2},\ ||Y+Y'||_{C_0^r}\leq\frac{\delta}{2}.\]
\end{prp}
\begin{proof}
The proof is almost identical to that of (\cite{MS04}, Theorem 3.5.2). We explain the setup and state the necessary quadratic estimate, leaving the rest to the enthusiastic reader. We first observe that the norms $W^{1,p}_Y$ (or $L^p_Y$) are Lipschitz equivalent for different $Y$ by compactness of $Z$ and that the Lipschitz coefficient $\ell_0$ can be chosen uniformly since the ball $B_{\epsilon}$ is compact. For each $Y$ let $r_Y$ denote the injectivity radius of $g_Y$ and set
\[I:=\min_{Y\in B_{\epsilon}}\left(\frac{r_Y}{100}\right)\]
(We are perhaps overly cautious, but for the proof of the quadratic estimate we will need to work deep inside a geodesic ball for a varying metric). Now for any $\xi\in W^{1,p}_0(u^*TZ)$ with $||\xi||_{W^{1,p}_0}<\epsilon_1=\frac{I}{c_0\ell_0}$ the exponential map
\[\xi\mapsto\exp^{g_Y}_u(\xi)\]
is an injective continuous map
\[V:=\{\xi\in W^{1,p}(u^*TZ) : ||\xi||_{W^{1,p}_0}<\epsilon_1\}\}\rightarrow\mB_Y\]
whose image we denote by $\nu_Y$. Write $\nu:=\bigcup_{Y\in B_{\epsilon}}\nu_Y$ and observe that exponentiation gives a trivialisation
\[\exp:V\times K\rightarrow\nu\]
There is a natural Banach bundle $\mE$ over $\nu$ whose fibre at $v\in\nu_Y$ is
\[\mE_{(v,Y)}:=L_Y^p(\Sigma,\Lambda^{0,1}\otimes_{J_Y}v^*TZ)\]
We must now trivialise this bundle compatibly with $\exp$. First we use parallel transport along geodesics using the $J_Y$-Hermitian connection $\tilde{\nabla}^Y$ associated to the Levi-Civita connection $\nabla^{g_Y}$ to construct isomorphisms
\[\Phi_{(v,Y)}:\mE_{(u,Y)}\rightarrow\mE_{(v,Y)}\]
We must still trivialise in the $Y$-direction. To this end we fix a smooth vector field $X$ on $\Sigma$ which vanishes on a set of measure zero. Recall that $X$ and $\alpha\in L_0^p(u^*TZ)$ together determine $\eta\in \mE_{(u,Y)}$ by the condition that $\eta(X/|X|)=\alpha$ almost everywhere (since then $\eta(jX/|X|)=-J_Y\alpha$ almost everywhere) - $L^p_Y$ integrability follows from $L^p_0$-integrability by Lipschitz equivalence of the norms. This gives an isomorphism $\psi_Y:\mE_{(u,0)}\rightarrow\mE_{(u,Y)}$ for all $Y$. Now the compositions $\psi_Y\circ\Phi_{(v,Y)}:E:=\mE_{(u,0)}\rightarrow\mE_{(v,Y)}$ give a trivialisation of the bundle $\exp^*\mE\cong E\times V\times K\rightarrow V\times K$ compatible with the diffeomorphism $V\times K\rightarrow\nu$.

The natural section $\dbar:\nu\rightarrow\mE$ taking $(u,J)$ to $\dbar_J(u)$ pulls back to a section $\mF:V\times K\rightarrow E$ of the trivialisation which we consider as a function between Banach spaces. We observe that
\[d_{(0,0)}\mF(\xi,Y)=D_u\xi+\frac{1}{2}Y\circ du\circ j\]
The key step in proving the implicit function theorem is the quadratic estimate:
\begin{lma}
In the setting of Proposition \ref{IFT}, there exists a constant $C>0$ such that the following holds for every volume form $d\vol_{\Sigma}$ with $c_p(d\vol_{\Sigma})\leq c_0$. If $||\xi||_{L^{\infty}_0}\leq c_0$ then
\[||d_{(\xi,Y)}\mF-D_u-\frac{1}{2}Y\circ du\circ j||\leq C\left(||\xi||_{W^{1,p}_0}+||Y||_{C_0^r}\right)\]
where the norm on the left is the operator norm.
\end{lma}
The proof of the proposition now follows precisely the same lines as (\cite{MS04}, Theorem 3.5.2).
\end{proof}
\begin{rmk}
When $u$ is a stable curve modelled on a bubble tree $T$ there is an exactly analogous statement which asserts existence and uniqueness of nearby stable curves modelled on the same bubble tree. We use the notation $\mJ_T$ for the space of $C^r$-smooth domain-dependent $\omega$-compatible complex structures whose domain is modelled on a bubble tree $T$. We will also write $T'<T$ to indicate that a bubble tree $T'$ is obtained from $T$ by merging bubbles (and hence decreasing the number of edges).
\end{rmk}
Employing the notation of the proof of our implicit function theorem we make the following observation.
\begin{sch}\label{handy}
If $u$ is a $J_0$-holomorphic stable map and $\kappa:B_{\epsilon}\rightarrow\mJ_T$ is an $\epsilon$-perturbation at $J_0$ such that $\iota_u(K)$ is a complement for $\im(D_u)$ then there is a small ball $0\in U\subset V\times K$ such that $\mF^{-1}(0)\cap U$ is the image of a $C^r$-smooth map $\ker(D_u)\oplus 0=\ker(D_u+\iota_u)\rightarrow V\times K$ of the form
\[(\xi,0)\mapsto(\xi,0)+Q\phi(\xi)\]
This is precisely the space of stable maps near $u$ which are $J_Y$-holomorphic for some $Y$ near 0.
\end{sch}
Before we state the gluing theorem we introduce some further notation. Let $T$ and $T'$ be bubble trees with $T'<T$: recall that a bubble tree consists of a configuration of marked domains $\Sigma_i$, $i\in I$ where some of the marked points are called nodes. By a node $n$ we mean a quadruple $(\Sigma_{i(n)},\Sigma_{j(n)},z_{i(n)},z_{j(n)})$ consisting of two (different) domains $\Sigma_{i(n)}$ and $\Sigma_{j(n)}$ and marked points $z_{i(n)}\in\Sigma_{i(n)}$, $z_{j(n)}\in\Sigma_{j(n)}$. We write $N$ for the set of nodes and think of the domain of the bubble tree as
\[\Sigma_T=\bigcup_{i\in I}\Sigma_i/(z_{i(n)}\sim z_{j(n)}:n\in N)\]
Label the nodes which are merged in going from $T$ to $T'$ by $M\subset N$. For each $n\in M$ define $A_n=T_{z_{i(n)}}\Sigma_{i(n)}\otimes_{\CC}T_{z_{j(n)}}\Sigma_{j(n)}$.

We will now construct a metric $g'$ on $\Sigma_{T'}$ given a metric $g$ on $\Sigma_{T}$ and describe how the complex structure changes (see \cite{McD}). By a metric we mean a smooth K\"{a}hler metric on each component. This will depend on a choice of $a_n\in A_n$ for each $n\in N'$; we denote this choice by $\underline{a}$ and call it \emph{gluing data}. Assume that $g$ is flat in a neighbourhood of $z_{i(n)}$ and $z_{j(n)}$ for each $n\in M$ and let $\exp_{i(n)}:T_{z_{i(n)}}\Sigma_{i(n)}\rightarrow\Sigma_{i(n)}$, $\exp_{j(n)}:T_{z_{j(n)}}\Sigma_{j(n)}\rightarrow\Sigma_{j(n)}$ denote the exponential maps. Using the map
\[\psi_n:T_{z_{i(n)}}\Sigma_{i(n)}\setminus\{0\}\rightarrow T_{z_{j(n)}}\Sigma_{j(n)}\setminus\{0\},\ \psi_n(x)=\frac{a_n}{x}\]
we can glue the domains $\Sigma_{i(n)}\setminus\exp_{i(n)}(B_{\sqrt{|a_n|}})$ and $\Sigma_{j(n)}\setminus\exp_{j(n)}(B_{\sqrt{|a_n|}})$ via the \emph{merging identification} $\exp_{j(n)}\circ\psi_n\circ\exp_{i(n)}^{-1}$. Choose a function $\chi_n:(0,\infty)\rightarrow (0,\infty)$ such that $\chi_n(s)=1$ when $s$ is slightly larger than $\sqrt{|a_n|}$ and such that $\chi_n(|x|)|dx|^2$ is $\psi_n$-invariant (note that $\psi_n^*(|dx|^2)=\left|\frac{a_n}{x}\right|^2|dx|^2$). The metric $g'$ is defined using this invariant metric to extend $g$ over the necks introduced by merging bubbles.

We also need to define how the family of almost complex structures $\kappa$ changes under gluing.
\begin{dfn}
Given bubble trees $T'<T$ and an $\epsilon$-perturbation $\kappa:B_{\epsilon}\rightarrow\mJ_T$ centred at $J_0$ we say $\kappa$ is {\em gluable} if there exists a constant $\mu$ such that on a $\mu$-neighbourhood of each $n\in M$ in the domain the almost complex structure $J_Y$ is domain-independent for all $Y\in B_{\epsilon}$. Given a gluable $\epsilon$-perturbation and a bubble tree $T'<T$ and a choice of $\underline{a}$ we define its {\em $\underline{a}$-gluing} to be the $\epsilon$-perturbation $\kappa_{\underline{a}}:B_{\epsilon}\rightarrow\mJ_{T'}$ where, for $z\in\Sigma_{T'}$, $\kappa_{\underline{a}}(Y)(z)$ equals $\kappa(Y)(\tilde{z})$, where $\tilde{z}\in\Sigma_T$ is sent to $z$ under the merging identification. Note that this is well-defined whenever $|a_n|<\mu$ for all $n\in M$ because $\kappa$ is gluable.
\end{dfn}
We do not give a proof of gluing but refer the reader to (\cite{MS04}, Chapter 10) for a detailed proof without varying $J$ and \cite{McD} for a less detailed proof with varying $J$.
\begin{prp}[Gluing]\label{glue}
Let $u$ be a stable $J_0$-holomorphic curve modelled on a bubble tree $T$ and suppose that $\kappa:B_{\epsilon}\rightarrow\mJ_T$ is a gluable $\epsilon$-perturbation centred at $J_0$ with the further property that $\iota_u(d\kappa(K))$ is a complement for $\im(D_u)$. Let $\mM_T(\kappa,c)$ denote the moduli space of stable maps modelled on the bubble tree $T$ which are $J_Y$-holomorphic for some $Y\in B_{\epsilon}$ and satisfy $||du_i||_{L^{\infty}_0}\leq c$ for all components $u_i$. Then for any $T'<T$ there is an $\epsilon'<\epsilon$, a neighbourhood $\nu$ of $u$ in the space of stable maps modelled on $T$, a non-increasing function $(0,\infty)\rightarrow (0,1):c\mapsto r(c)$ and, for each $(c,\underline{a})$ with $0<|a_n|<r(c)^2$, $c>||du||_{L^{\infty}_0}$ an embedding
\[\Gl(u,\underline{a},c):(\nu\times B_{\epsilon'})\cap\mM_T(\kappa,c)\rightarrow\mM_{T'}(\kappa_R,\infty)\]
with the obvious analogues of properties (i)-(iv) in (\cite{MS04}, Theorem 10.1.2). In particular (\cite{MS04}, Corollary 10.1.3)
\begin{itemize}
\item[($\dag$)] if $r_j\rightarrow\infty$, $\underline{a}_j$ is a sequence of gluing data with $|a_{j,n}|<r_j$ and $u_j\in\mM_{T'}(\kappa_{j},\infty)$ is a sequence which Gromov-converges to $u$ then there is a $j_0$ such that for all $j\geq j_0$ we have $u_j\in\im(\Gl(u,\underline{a},c))$.
\end{itemize}
\end{prp}
We now return to the proof of Theorem \ref{GWvanish}. We will denote by $\mMuniv_T$ the universal moduli space of pseudoholomorphic stable maps modelled on a tree $T$ representing the class $\beta$, considered as a subset of $\mB_T$. Here $\mB_T:=\bigcup_{J\in\mJ}\mB_{J,T}$ denotes the union over $J\in\mJ$ of the space of $W^{1,p}_J$-maps modelled on a bubble tree $T$.
\begin{proof}
Consider a smooth 1-parameter family $J_t\in\mJ$ of domain-{\em independent} almost complex structures such that $J_0=J_-$ and $\left.\frac{\partial}{\partial t}\right|_{t=0}J_t=\delta J$. Assume that the Gromov-Witten invariant is nonvanishing so that for each $t\in[0,1]$ there exists a $J_t$-holomorphic stable map $u'_t$ in the class $\beta$. By Gromov compactness we can extract a subsequence $0<t_k\rightarrow 0$ such that $u_{t_k}=u'_{t_k}\circ\phi_k$ converges (for some sequence of reparametrisations $\phi_k$) to a stable $J_0$-holomorphic map $u$. We may artificially add marked points to the domain to ensure there are no automorphisms.

Pick an extension of $J_t$ to an $\epsilon$-perturbation at $J_-$
\[\kappa:B^{\exc}_{\epsilon}\rightarrow\mJ\]
(where $B_{\epsilon}^{\exc}$ is an $\exc$-dimensional Euclidean ball) with the property that $K:=\im(d_0\kappa)\subset T_{J_-}\mJ$ satifies
\[\iota_u(K)\cap\im(D_u)=\{0\}\]
The existence of this perturbation is precisely the transversality theorem (\cite{MS04}, Theorem 3.2.1) applied to each (possibly non-simple) component of the stable curve $u$ - transversality can be achieved for the non-simple components by allowing domain-dependent $J$s since we are in a semi-positive symplectic manifold. Note that to {\em define} Gromov-Witten invariants we cannot achieve this transversality simultaneously over all strata because the bubbles that develop have domain-independent almost complex structures. That is not our goal: we wish to find a contradiction to the existence of the particular Gromov-convergent sequence $u_k$ we constructed under the assumption that the Gromov-Witten invariant was non-zero.

By passing to a subsequence we can assume that all $u_k$ are modelled on the same bubble tree $T'$. Let us first assume that $u$ is also modelled on $T'$. Observe that by construction $D_u+\iota_u$ is surjective and Fredholm and therefore admits a right inverse $Q_u$. Scholium \ref{handy} tells us that there is a neighbourhood $U$ of $(u,J_-)$ such that all $J_Y$-holomorphic stable maps $\exp^{g_Y}_u(\xi)$ with $(\xi,Y)\in U$ are in the image of a $C^r$-smooth map
\[(\xi,0)\mapsto(\xi,0)+Q_u\phi(\xi)\]
defined on the kernel $\ker(D_u)\oplus 0=\ker(D_u+\iota_u)$. For $k$ large enough $u_k$ is $J_{Y_k}$-holomorphic (for some $Y_k\neq 0$) and of this form since the sequence converges to in the $C^{r}$-topology to $u$. However, we know that $\mM_T(J_-)$ is a smooth manifold near $u$ with tangent space $\ker(D_u)$. This implies that $Q_u(\phi(\xi)\in V\times \{0\}$ for $\xi$ small enough. This contradicts the fact that $Y_k\neq 0$. Therefore $u$ must be modelled on a different bubble tree, $T'<T$.

We choose a gluable $\epsilon$-perturbation $\kappa:B^{\exc}_{\epsilon}\rightarrow\mJ_{T'}$ centred at $J_-$ which contains the family $J_t$ and which satisfies $\iota_u(d\kappa(K))\cap\im(D_u)=0$. The existence of this perturbation is guaranteed by transversality, for we can pick the almost complex structures domain-dependently away from the $\mu$-neighbourhoods of the nodes. Since our original family $J_t$ consisted of domain-independent almost complex structures the $\underline{a}$-gluing $\kappa_{\underline{a}}$ contains $J_t$ for all $\underline{a}$ small enough. Pick $c>||du||_{L^{\infty}_0}$. By assumption the gluing map $\Gl(u,\underline{a},c)$ lands in $\mM_T(J_-,\infty)$ for $\underline{a}$ small enough. By Proposition \ref{glue} $(u_k,J_k)$ lies in the image of $\Gl(u,\underline{a},c)$ for large $k$, however $J_k\neq J_-$ by assumption. This is a contradiction.
\end{proof}
\begin{rmk}\label{SalEtAl}
It remains to show that Property ($\dag$) of Proposition \ref{glue} holds for our moduli spaces of $J_-$-holomorphic curves in the twistor fibre. This follows from \cite{RRS}, since $J_-|_F$ is an integrable complex structure on the twistor fibre and the twistor fibre is a convex manifold (i.e. all holomorphic curves are regular). The main theorem of \cite{RRS} therefore implies that the moduli space of stable maps into each twistor fibre is a smooth orbifold. In particular it tells us that in the neighbourhood of a stable $J_-$-holomorphic curve $u$ with no automorphisms the moduli space is a smooth manifold of dimension $\exc$ (which is equal to the expected dimension of $u$ considered as a curve in the twistor fibre). In the cases we need it is easy to check that the strata are all still smooth upon intersecting with the explicit cycles $X_1,\ldots,X_k$.
\end{rmk}
\section{Floer theory of Reznikov Lagrangians}\label{FF}
\subsection{Obstruction term}
The following notion was introduced by Fukaya, Oh, Ohta and Ono in a more general context in their book \cite{FOOO}. It arises as the first of an infinite sequence of filtered $A_{\infty}$ operations on a suitable space of singular chains on $L$.
\begin{dfn}
Let $L$ be a monotone Lagrangian. If $J$ is regular for all moduli spaces of Maslov 2 discs then the obstruction $\m_0$ is the chain represented by the evaluation map from the moduli space of Maslov 2 discs with a single boundary marked point to $L$.
\end{dfn}
In our case ($n=3$) there is precisely one component in this moduli space, corresponding to the hemispheres of real algebraic lines in $\CC\PP^3$ (with boundary on $SO(3)\cong\RR\PP^3$). The expected dimension of the moduli space is $n+\mu-3+1=6=\dim(L_{\Sigma})$ so the obstruction cycle is homologous to a multiple of the fundamental class. Let us write $\FF_{\mu,k}$ for the Fukaya-Floer chain
\[\ev:\mM_{\mu,k}\rightarrow L_{\Sigma}^k\]
where $\mM_{\mu,k}$ denotes the moduli space of Maslov-$\mu$ discs with boundary on $L_{\Sigma}$ and $k$ boundary marked points. We now prove the first part of Theorem \ref{obs-fooo}.
\begin{thm}\label{obs-bit}
If $\Sigma$ is an oriented totally geodesic submanifold of an oriented hyperbolic 6-manifold $M$ and $L_{\Sigma}$ denotes the Reznikov Lagrangian lift in the twistor space of $M$ then
\[\m_0=\pm 2\sqrt{q}[L_{\Sigma}].\]
\end{thm}
\begin{proof}
The moduli space of Maslov 2 discs (twistor hemispheres) is compact so we can employ obstruction bundle techniques to compute $\FF_{2,1}$ (note that by definition $[\FF_{2,1}]=\m_0$). Since the totally geodesic submanifold $\Sigma$ is oriented it has the form
\[\Gamma_{\Sigma}\backslash SO^+(3,1)\times SO(3)/SO(3)\times SO(3)\]
and the Reznikov lift is
\[\Gamma_{\Sigma}\backslash SO^+(3,1)\times SO(3)/SO(3)_{\Delta}\]
where $SO(3)_{\Delta}=SO(3)\times SO(3)\cap U(3)$ is the diagonal subgroup. A twistor line with boundary on $L_{\Sigma}$ can be specified by giving a unit vector $v\in T_p\Sigma$ and a unit normal vector $w\in\nu_p\Sigma$ and taking the set of $\psi$ preserving the 2-plane $\langle v,w\rangle$. Hence the moduli space of twistor hemispheres is a $S^2\times S^2=\tilde{\mathrm{Gr}}_1(\RR\PP^3)$-bundle over $\Sigma$
\[\Gamma_{\Sigma}\backslash SO^+(3,1)\times SO(3)/(SO(2)\times SO(2))\]
Adding a marked point on the boundary we obtain
\[\Gamma_{\Sigma}\backslash SO^+(3,1)\times SO(3)/SO(2)_{\Delta}\]
The linear analysis of the $\dbar$-operator is identical to the case of closed curves except that we only allow deformations which come from vector fields on $\Sigma$ (that is, $\mH=\tau^*(TM)$ is replaced by $\tau|_{L_{\Sigma}}^*(T\Sigma)$). This implies that the obstruction bundle is 2-dimensional with Euler class equal to the Euler class of the tautological $SO(2)$-bundle. The 1-point invariant $\FF_{2,1}$ is therefore given by evaluating the fibre integral of the this Euler class which gives $\pm 2$. The $\pm 1$ comes from the choice of spin structure on $L_{\Sigma}$ (or alternatively from picking a flat connection with holonomy $\pm 1$ around the nontrivial loop in the $SO(3)$ factor) and the $\sqrt{q}$ in the formula for $\m_0$ comes from the area of holomorphic discs (if $q=\exp(-\int_A\omega)$ for a twistor line $A$ then $\sqrt{q}=\exp(-\int_h\omega)$ for a hemisphere $h$).

Note that $L_{\Sigma}$ admits a spin structure since it is diffeomorphic to $\Sigma\times SO(3)$ both factors of which are spin, the principal frame bundle of $\Sigma$ being trivial since $\Sigma$ is an orientable 3-manifold. Changing the spin structure along the nontrivial loop in $SO(3)$ changes the sign of $\m_0$ while changing the spin structure along a loop from $\Sigma$ has no effect (there are no discs with such a boundary).
\end{proof}
\subsection{Quantum homology of Reznikov Lagrangians}
We finish the proof of Theorem \ref{obs-fooo} by calculating the quantum homology of a Reznikov Lagrangian using the techniques and definitions of \cite{BC}. We recall that the quantum homology $QH(L)$ of an oriented monotone Lagrangian submanifold $L$ is defined to be the homology of the pearl complex associated to a choice of Morse function $F$ on $L$, metric on $L$ (such that the gradient flow is Morse-Smale) and generic almost complex structure on $Z$. The chain groups are the free $\CC$-modules on the critical points of $F$ and the differential counts oriented ``pearly trajectories'' which are sequences of $F$-gradient flowlines and $J$-holomorphic discs with boundary on $L$. Although the theory is developed in \cite{BC} with $\ZZ/2$-coefficients the orientation issue is cleared up in (\cite{BC2}, Appendix A).
\begin{thm}
Let $\Sigma$ be an oriented totally geodesic 3-dimensional submanifold of an oriented hyperbolic 6-manifold. The quantum homology of the Reznikov lift $L_{\Sigma}$ is
\[QH_*(L)\cong H_*(L;\CC[t])\]
where we write $t=q^{1/2}$ for the Novikov parameter.
\end{thm}
We recall that the quantum homology of a monotone Lagrangian is (noncanonically) isomorphic to its self-Floer homology so this proves Theorem \ref{obs-fooo}.
\begin{proof}
We pick a Morse function $f$ on $\Sigma$ and the standard Morse function $r$ with four critical points on $\RR\PP^3$. We assume that the gradient flow on $L_{\Sigma}\cong \Sigma\times SO(3)$ of the function $F=f+r$ with respect to the product of the hyperbolic and round metrics is Morse-Smale (by suitable choice of $f$) and that $f$ (and hence $F$) has a unique maximum and a unique minimum. We may also assume that $f$ is self-indexing.

There is (\cite{BC}, Proposition 6.1.1, Proof A) a homology spectral sequence whose $E^1$-page is
\[E^1_{i,j}=H_{i-j}(L_{\Sigma};\CC)t^{-j}\]
and which converges to the quantum homology of $L_{\Sigma}$. We draw the $E^1$ page below (denoting $b_1(\Sigma)=:b$) and indicate the differentials we will show to be zero.
{\begin{center}
\begin{tikzpicture}
  \matrix (n) [matrix of math nodes,
    nodes in empty cells,nodes={minimum width=2ex,
    minimum height=2ex,outer sep=-2pt},
    column sep=1ex,row sep=1ex]{
0 & 0 & 0 & \CC\\
0 & 0 & \CC t & \CC^b\\
0 & \CC t^2 & \CC^bt & \CC^b\\
\CC t^3 & \CC^bt^2 & \CC^bt & \CC^2\\
\CC^bt^3 & \CC^bt^2 & \CC^2t & \CC^b\\
\CC^bt^3 & \CC^2t^2 & \CC^bt & \CC^b\\
\CC^2t^3 & \CC^bt^2 & \CC^bt & \CC\\
\CC^bt^3 & \CC^bt^2 & \CC t& 0\\
\CC^bt^3 & \CC t^2& 0 & 0\\
\CC t^3 & 0 & 0 & 0\\};
\draw[-stealth] (n-7-4.west) -- (n-7-3.east);
\draw[-stealth] (n-7-4.west) -- (n-6-2.south east);
\draw[-stealth] (n-7-4.west) -- (n-5-1.east);
\draw[-stealth] (n-6-4.west) -- (n-6-3.east);
\draw[-stealth] (n-6-4.west) -- (n-5-2.east);
\draw[-stealth] (n-5-4.west) -- (n-5-3.east);
\end{tikzpicture}
\end{center}}
That the other differentials vanish follows either for degree reasons or by a combination of Poincar\'{e} duality and the Leibniz property of the higher differentials with respect to cup product (which is certainly true on the $E^1$-page and continues to be true on the $E^2$ and $E^3$ pages because the $E^1$ and $E^2$ differentials vanish). Let us write $\delta^r_{i,j}$ for the $r$-th differential whose domain is $E^r_{i,j}$.

The differentials all vanish, but not all for the same reason. We now tackle the reasons the differentials vanish case by case.

\par\textbf{Filtering by the value of} $f$\textbf{:} We observe that the contribution to the higher differentials from pearly trajectories joining a critical point $p$ to a critical point $q$ vanishes when $f(p)>f(q)$. To see this, let $J_k$ be a sequence of regular almost complex structures with $J_k\rightarrow J_-$ as $k\rightarrow\infty$ and suppose to the contrary that there is a nonzero differential connecting $p$ to $q$. Let $u_k$ be a pearly trajectory contributing to this differential. We can extract a convergent subsequence $u_{k'}$ and the limit is a broken pearly trajectory whose discs are $J_-$-holomorphic and therefore contained in level sets of the function $f$. Since the gradient flow decreases $f$ and the discs do not allow one to return to larger values of $f$ we see that $f(p)>f(q)$. This argument proves vanishing of $\delta^1_{0,0}$, $\delta^1_{0,1}$ and $\delta^3_{0,0}$.

\par\textbf{Easy obstruction bundle methods:} To prove that $\delta^1_{0,2}=0$ notice that this differential counts (for a regular $J$) pearly trajectories connecting a critical point $y$ of index 2 (which has index 2 as a critical point of $f$ and 0 as a critical point of $r$) to the critical point $q$ of index 3 corresponding to the maximum of $r$ and the minimum of $f$. There is precisely one $J$-holomorphic disc in this trajectory and it has Maslov index 2. Such a pearly trajectory corresponds precisely to a $J$-disc whose boundary intersects the unstable manifold of $y$ and the stable manifold of $q$. Since the moduli space of Maslov 2 discs is compact by minimality of the relative homology class the Fukaya-Floer chain $\FF_{2,2}$ of Maslov 2 discs with two boundary marked points is a cycle: its boundary has two types of component, where the first marked point approaches the second from a clockwise or an anticlockwise direction. These cancel so the boundary of $\FF_{2,2}$ is the zero chain. By a priori compactness of the moduli space we can compute the homology class of this cycle using the obstruction bundle techniques we used to prove Theorems \ref{GWbigformula-thm} and \ref{obs-bit}. The moduli space $\mM_{2,0}$ is diffeomorphic to $S^2\times S^2\times\Sigma$ so if we write $x,y\in H^2(S^2\times S^2;\ZZ)$ for a $\ZZ$-basis with $x^2=y^2=0$ then the diagonal decomposition for $\Delta^2:\mM_{2,0}\rightarrow\mM_{2,0}^2$ is
\[\{x\otimes y+1\otimes xy\}\cup\tau^*\Delta_!^2(1_{\Sigma})\]
The Euler class of the obstruction bundle is $x$ (which is the Euler class of one of the two tautological $SO(2)$-bundle over $S^2\times S^2=SO(4)/SO(2)\times SO(2)$). Cupping the diagonal with $x\otimes 1$ gives
\[\{xy\otimes x)\}\cup\tau^*\Delta_!^2(1_{\Sigma})\]
In each fibre the moduli space of Maslov 2 discs with one marked point $\mM_{2,1}(F)$ is a circle bundle over $S^2\times S^2$ with Euler class $x+y$ and (by the Gysin sequence) it has
\[H^4(\mM_{2,1}(F);\ZZ)=0\]
Therefore the pullback of $xy$ to $\mM_{2,1}$ vanishes. In particular the Fukaya-Floer cycle is nullhomologous. Since the count of pearly trajectories contributing to this differential is just the intersection number of this cycle with the product of the stable manifold of $q$ and the unstable manifold of $y$, the differential vanishes.

\par\textbf{Hard obstruction bundle methods:} To prove that $\delta^2_{0,0}=0$ notice that this differential counts (for a regular $J$) pearly trajectories going from the global minimum $p$ of $F$ to the critical point $q$ of index 3 corresponding to the maximum of $r$ and the minimum of $f$. (That the differential has no contribution from pearly trajectories going from $p$ to the critical point $q'$ corresponding to the maximum of $f$ and the minimum of $r$ follows by the previous filtering argument.) Such trajectories consist of a (Maslov 4) $J$-holomorphic disc through the global minimum whose boundary intersects the stable manifold of $q$. Assume that the differential is nonzero and that therefore such pearly trajectories exist for any $J$ and suppose that $J_t$ is a family of almost complex structures obtained by exponentiating an infinitesimal deformation $\delta_vJ$ at $J_-$ associated to some vector field $v$ on $M$ as in the proof of Theorem \ref{GWvanish}. There is a sequence of $J_{t_i}$-holomorphic pearly trajectories ($J_{t_i}\rightarrow J_-$) which converges to some limit trajectory $u$. Since $J_-$-holomorphic discs are restricted to lie within a single twistor fibre and the stable manifold of $q$ intersects the fibre containing $p$ and $q$ precisely at $q$ we know that $u$ is just a stable Maslov 4 disc. Now a gluing or implicit function theorem argument as in the proof of Theorem \ref{GWvanish} (modified as in Section 4 of \cite{BC} to the case of holomorphic discs) shows that, for suitable choice of $v$, the $J_{t_i}$-holomorphic pearly trajectories cannot exist for $i$ large enough. A similar argument proves $\delta^1_{0,1}=0$.
\end{proof}

As was remarked in the introduction, this tallies with the fact that a Lagrangian with nonvanishing self-Floer cohomology has obstruction term equal to an eigenvalue of the first Chern class acting by quantum product on the quantum cohomology. It would be intriguing to find (or to rule out the existence of) monotone Lagrangians in the twistor space of a hyperbolic 6-manifold whose $\m_0$ equals one of the four ``exotic'' eigenvalues involving the Euler characteristic of $M$.


\begin{thebibliography}{1}
	\bibitem{Aur} D. Auroux, `Mirror symmetry and T-duality in the complement of an anticanonical divisor', {\em Journal of G\"{o}kova Geometry Topology}, Volume 1 (2007) 51--91, MR2386535, Zbl 1181.53076.
	\bibitem{BC} P. Biran and O. Cornea, `Quantum structures for Lagrangian submanifolds', preprint arXiv:0708.4221, 2007.
	\bibitem{BC2} P. Biran and O. Cornea, `Lagrangian topology and enumerative geometry', {\em Geometry and Topology}, Volume 16 (2012) Number 2, 963--1052, MR2928987, Zbl 1253.53079.
	\bibitem{BH} A. Borel and F. Hirzebruch, `Characteristic classes and homogeneous spaces, II' {\em American Journal of Mathematics}, 81 (1959) 315--382, MR0110105, Zbl 0097.36401.
	\bibitem{CT} J. A. Carlson and D. Toledo, `Harmonic mappings of K\"{a}hler manifolds to locally symmetric spaces' \emph{Publications Math\'{e}matiques d'IHES}, 69 (1989) 173--201, MR1019964, Zbl 0695.58010.
	\bibitem{Chern} S.-S. Chern, `On curvature and characteristic classes of a Riemannian manifold', {\em Abhandlungen 
aus dem Mathematischen Seminar der Universit\"{a}t Hamburg}, 20 (1955), 117--126, MR0075647, Zbl 0066.17003.
	\bibitem{DavidovMuskarovActaMathHungarica} J. Davidov and O. Muskarov, `On the Riemannian curvature of a twistor space', {\em Acta Mathematica Hungarica}, 58 (1991), no. 3-4, 319--332. MR1153487, Zbl 0747.53030.
	\bibitem{ES} J. Eells and S. Salamon, `Twistorial constructions of harmonic maps of surfaces into four-manifolds', \emph{Annali della Scuoli Normale Superiore di Pisa, Classe di Scienze $4^e$ s\'{e}rie}, tome 12, no. 4 (1985) 589--640, MR0848842, Zbl 0627.58019.
	\bibitem{ESam} J. Eells, Jr. and J. H. Sampson, `Harmonic Mappings of Riemannian Manifolds', {\em American Journal of Mathematics}, Vol. 86, No. 1 (1964), 109--160, MR0164306, Zbl 0122.40102.
	\bibitem{FP} J. Fine and D. Panov `Symplectic Calabi-Yau manifolds, minimal surfaces and the hyperbolic geometry of the conifold' \emph{Journal of Differential Geometry}, Volume 82, Number 1 (2009), 155--205, MR2504773, Zbl 1177.32014.
	\bibitem{FP2} J. Fine and D. Panov, `Hyperbolic geometry and non-K\"{a}hler manifolds with trivial canonical bundle', \emph{Geometry and Topology} 14 (2010) 1723--1763, MR2679581, Zbl 1214.53058.
	\bibitem{FOOO} K. Fukaya, Y. G. Oh, H. Ohta and K. Ono,  `Lagrangian intersection Floer
theory: anomaly and obstruction. Part I and II', {\em AMS/IP Studies in Advanced Mathematics}, 46.1 and 46.2. American Mathematical Society, Providence, RI; International Press, Somerville, MA, 2009,  MR2553465, MR2548482, Zbl 1181.53002, Zbl 1181.53003.
	\bibitem{FO} K. Fukaya and K. Ono, `Arnold conjecture and Gromov-Witten invariant', {\em Topology}, 38, (1999) 933--1048, MR1688434, Zbl 0946.53047.
	\bibitem{OssermanETAL} R. D. Gulliver II, R. Osserman and H. L. Royden `A theory of branched immersions of surfaces', {\em American Journal of Mathematics}, 95 (1973) 750--812, MR0362153, Zbl 0295.53002.
	\bibitem{Jo} J. Jost, `Riemannian geometry and geometric analysis', Fourth edition. Universitext. Springer-Verlag, Berlin, 2005, MR2165400, Zbl 1083.53001.
	\bibitem{Kont} M. Kontsevich, `Enumeration of rational curves with torus actions', In: {\em The moduli space of curves}, (R. Dijkgraaf, C. Faber and G. van der Geer eds.) Progress in Mathematics 129, Birkh\"{a}user, Boston (1995) 335--368, MR1363062, Zbl 0885.14028.
	\bibitem{LaFRoy} J.-F. LaFont and R. Roy, `A note on the characteristic classes of non-positively curved manifolds', {\em Expositiones Mathematicae}, Volume 25 (2007) 21--35, MR2286832, Zbl 1151.57033.
	\bibitem{Massey} W. S. Massey `On the cohomology ring of a sphere bundle' Journal of Mathematics and Mechanics, Volume 7, Number 2 (1958) 265--289, MR0093763, Zbl 0089.39204.
	\bibitem{McD} D. McDuff, `Almost complex structures on $S^2\times S^2$', {\em Duke Mathematical Journal} 101 (2000) 135--177, MR1733733, Zbl 0974.53020.
	\bibitem{MS04} D. McDuff and D. Salamon, `J-holomorphic curves and symplectic topology', Colloquium Publications, vol. 52, American Mathematical Society, Providence, R.I., 2004, MR2045629, Zbl 1064.53051.
	\bibitem{TodaMimura} M. Mimura and H. Toda `Topology of Lie groups. I, II.' {\em Translations of Mathematical Monographs}, 91. American Mathematical Society, Providence, RI, 1991, MR1122592, Zbl 0757.57001.
	\bibitem{OkVan} C. Okonek and A. Van de Ven, `Cubic forms and complex 3-folds', {\em L'Enseignement Math\'{e}matique}, 41 (1995) 297--333, MR1365849, Zbl 0869.14018.
	\bibitem{Rawn} J. H. Rawnsley `f-structures, f-twistor spaces and harmonic maps', In: {\em Geometry seminar ``Luigi Bianchi'' II-1984}, Lecture Notes in Math., 1164, Springer, Berlin, (1985) 85--159, MR0829229, Zbl 0592.58009.
	\bibitem{Rez} A. Reznikov, `Symplectic twistor spaces', {\em Annals of Global Analysis and Geometry}, Volume 11, Number 2, 109--118, MR1225431, Zbl 0810.53056.
	\bibitem{RRS} J. W. Robbin, Y. Ruan and D. A. Salamon, `The moduli space of regular stable maps', {\em Mathematische Zeitschrift}, 259 (2008) 525--574, MR2395126, Zbl 1171.30018.
	\bibitem{Sal} S. Salamon, `Harmonic and holomorphic maps', In: {\em Geometry seminar ``Luigi Bianchi'' II-1984}, Lecture Notes in Math., 1164, Springer, Berlin, (1985) 161--224, MR0829230, Zbl 0591.53031.
	\bibitem{Th} E. Thomas, `Fields of tangent 2-planes on even-dimensional manifolds', {\em The Annals of Mathematics}, Second Series, Vol. 86, No. 2 (1967), 349--361, MR0212834, Zbl 0168.21401.
  \end{thebibliography}
\end{document}